\documentclass[11pt]{article} 

\usepackage[utf8]{inputenc} 
  
\usepackage{geometry} 
\usepackage{graphicx} 

\usepackage{booktabs} 
\usepackage{array} 
\usepackage{paralist} 
\usepackage{verbatim} 
\usepackage{amsmath}
\usepackage{amssymb}
\usepackage{amsthm}
\usepackage{xifthen}
\usepackage{xcolor}

\usepackage{subfigure}

\usepackage{fancyhdr} 
\usepackage{sectsty}
\allsectionsfont{\sffamily\mdseries\upshape} 

\newtheorem{theorem}{Theorem}

\newtheorem{definition}[theorem]{Definition}

\newtheorem{lemma}[theorem]{Lemma}

\newtheorem{proposition}[theorem]{Proposition}
\newtheorem{remark}[theorem]{Remark}

\usepackage[nottoc,notlof,notlot]{tocbibind} 
\usepackage[titles,subfigure]{tocloft} 

\newcommand{\E}{\mathbb E}

\newcommand{\F}{\mathcal F}
\newcommand{\I}[1]{\mathbb I_{\{#1\}}}

\renewcommand{\P}{\mathbb P}

\newcommand{\kom}[1]{}
\renewcommand{\kom}[1]{{\bf [#1]}}

\newcounter{komcounter}
\numberwithin{komcounter}{section}

\title{How to detect a salami slicer: a stochastic controller-stopper game with unknown competition}
\author{Erik Ekström\\Department of Mathematics, Uppsala University\\ \\ Kristoffer Lindensjö\\Department of Mathematics, Stockholm University\\ \\Marcus Olofsson\\Department of Mathematics, Uppsala University
}

\begin{document}

\maketitle

\begin{abstract}
We consider a stochastic game of control and stopping specified in terms of a process $X_t=-\theta \Lambda_t+W_t$, representing the holdings of Player~1, where $W$ is a Brownian motion, $\theta$ is a Bernoulli random variable indicating whether Player~2 is active or not, and $\Lambda$ is a non-decreasing process representing the accumulated ``theft'' or ``fraud'' performed by Player 2 (if active) against Player~1. Player 1 cannot observe $\theta$ or $\Lambda$ directly, but can merely observe the path of the process $X$ and may choose a stopping rule $\tau$ to deactivate Player~2 at a cost $M$. Player 1 thus does not know if she is the victim of fraud and operates in this sense under unknown competition. Player~2 can observe both $\theta$ and $W$ and seeks to choose the fraud strategy $\Lambda$ that maximizes the expected discounted amount  
\[ \E \left [ \theta\int _0^{\tau} e^{-rs}  d\Lambda_s \right ],\]
whereas Player 1 seeks to choose the stopping strategy $\tau$ so as to minimize the expected discounted cost
\[\E  \left [\theta \int _0^{\tau} e^{-rs} d\Lambda_s   + e^{-r\tau}M\I{\tau<\infty} \right ].\]
This non-zero-sum game appears to be novel and is motivated by applications in fraud detection; it combines filtering (detection), non-singular control, stopping, strategic features (games) and asymmetric information. We derive Nash equilibria for this game; for some parameter values we find an equilibrium in pure strategies, and for other parameter values we find an equilibrium by allowing for randomized stopping strategies.
\end{abstract}

\section{Introduction}\label{intro}
``Salami slicing'' is a well known concept of fraud, where a thief or a fraudster repeatedly steals small amounts - so small that the net loss from the system is hard to detect at each instant. Albeit a small intensity of theft, aggregated over time or over a large number of victims, the accumulated amount stolen may become large if not detected in time.
The essence of a salami slicing strategy is thus to trade-off between maximizing short-time profits and avoiding detection. 
Classical examples include penny shaving, i.e., consistently rounding off a large number of transactions up to the nearest penny and stealing the difference (cf. \cite{K}), modifying the fuel tank capacity of rental cars so that customers pay for slightly more fuel than they consume at each rental (see \cite{BKW}),
and stealing one coin from each fare box of a public transportation agency (see \cite{H}).
Further examples in the same spirit can easily be thought of in the context of, e.g., poaching of game or fish (where the poacher would trade-off between an offensive strategy and being detected), a spy who steals classified information (again, a balance between exploitation and detection is a natural ingredient),
computer networks with possible information leakage or botnet defense (cf. \cite[Sec. 7.1.10]{CD} and \cite{KB} where mean field games of botnet defense are studied). Yet another related example is in the context of online privacy issues, where individuals knowingly
give away small pieces of information about themselves, but would refuse to do so if they knew that this information was fully exploited against them.

An essential ingredient in a salami slicing strategy is the possibility to avoid detection by the presence of natural stochastic fluctuations in the underlying system; an observed fluctuation may then be due to such stochasticity, or stem from a fraudster being active (or from a combination of both). In this paper, we model the natural fluctuations by a Brownian motion and assume that the account holder, i.e., the possible victim and fraud detector,  can only observe the total effect of natural fluctuations and the accumulated theft. Moreover, we equip the account holder with the possibility to deactivate the potential fraudster at any time at a given cost $M$. As discussed above, there is then a natural trade-off for the fraudster between stealing large amounts and being detected. From the perspective of fraud detection on the other hand, the account holder has to balance the losses of potential fraud
against the cost of deactivation.

\subsection{Mathematical problem formulation}

Let $(\Omega, \F, \P)$ be a probability space on which a standard Brownian motion $W = (W_t)_{t\geq 0}$ and an independent Bernoulli random variable $\theta$ with $\P(\theta=1)=p=1-\P(\theta=0)$ are given, where $p\in(0,1)$. 
We consider a two-player stochastic game in which the players are referred to as ``account holder'' (Player $1$) and ``fraudster'' (Player $2$). Let the stochastic process $X= (X_t)_{t\geq 0}$ be given by
\begin{equation*}\label{eq:SDE}
X_t =- \theta \Lambda_t + W_t,
\end{equation*} 
where $\Lambda= (\Lambda_t)_{t\geq 0}$ is a non-decreasing stochastic process chosen by the fraudster. The interpretation is that $X$ represents the wealth of the account holder, $\theta$ is an indicator of whether the fraudster is active ($\theta=1$) or not ($\theta=0$), $p$ represents the account holder's initial belief that the fraudster is active, and $\Lambda$ corresponds to the accumulated amount stolen by the fraudster (if active).

We assume that the strategy $\Lambda$ of the fraudster is $\F^{W,\theta}$-adapted, where 
$\F^{W,\theta} = (\F^{W,\theta}_t)_{t\geq 0}$ is the augmented filtration generated by $W$ and $\theta$. 
The interpretation is that the fraudster knows whether he is active or not, as well as 
observes the natural fluctuations, i.e., $W$, in the account, and based on this information decides how to steal.

The account holder, on the other hand, is assumed not to have direct access to the underlying Brownian motion $W$, but can merely observe the path of the aggregate process $X$. The account holder is equipped with a stopping control as follows: 
at any time, she may choose to deactivate the fraudster by paying a fixed amount $M>0$, thereby ruling out 
the possibility for an active fraudster to continue stealing. The control of the account holder is thus represented by a random 
time. In Sections~\ref{intro}-\ref{sec:pure}, the account holder strategy will be chosen from the class of stopping times with respect to $\F^X = (\F^X_t)_{t\geq 0}$, which is the augmented filtration generated by $X$. Then, in Section~\ref{sec:mixed}, we will also consider randomized stopping times. 
%

%
%
 
\begin{definition} 
A continuous $\F^{W,\theta}$-adapted
non-decreasing process $\Lambda$  with $\Lambda_{0}=0$ is said to be a fraud strategy and an $\F^X$-stopping time $\tau$ is said to be a stopping strategy. The set of fraud strategies is denoted by $\mathbb L$, and the set of $\F^X$-stopping times 
is denoted by $\mathbb T$.
Given a pair $(\tau,\Lambda)\in\mathbb T\times \mathbb L$, the expected cost for the account holder is defined as 
\begin{equation} \label{E-cost}
\mathcal J^1(\tau,\Lambda;p) = \E \left [\theta \int _0^{\tau} e^{-rs} d\Lambda_s   + e^{-r\tau}M\I{\tau<\infty} \right ],
\end{equation}
and the expected payoff for the fraudster is defined as 
\begin{equation}\label{E-payoff}
\mathcal J^2(\tau,\Lambda;p) =  \E \left [ \theta\int _0^{\tau} e^{-rs}  d\Lambda_s \right ],
\end{equation}
where the discount rate $r$ is a positive constant. 
\end{definition}

\begin{remark}\label{precedence:rem}
For admissibility, we require that the fraud strategy $\Lambda$ is continuous.
An alternative specification would allow for  
strategies that are merely right-continuous with left limits, where then additional
care in the definition of ${\mathcal J}^1$ and ${\mathcal J}^2$ in the case of a lump sum payment and simultaneous stopping is needed. 
For example, replacing the upper limit of integration in \eqref{E-cost} and \eqref{E-payoff} by $\tau-$ would correspond to
a specification giving precedence to the stopper. However, since a jump in $X$ would immediately reveal the existence of the fraudster, it is 
(at least heuristically) clear that the account holder would stop at the first jump time, thereby reducing that set-up to the present case with only continuous fraud strategies.
\end{remark}

In the above set-up the account holder naturally seeks to choose a strategy $\tau$ to minimize the cost \eqref{E-cost}, whereas the fraudster seeks to choose a strategy $\Lambda$ to maximize \eqref{E-payoff}, and we define a Nash equilibrium 
accordingly.

\begin{definition}[Nash equilibrium]
\label{NE} 
A pair of strategies $(\tau^\ast,\Lambda^\ast)\in\mathbb T\times \mathbb L$ is a Nash equilibrium (NE) if it satisfies
\begin{equation}\label{NEineq}
\left\{\begin{array}{ll}
\mathcal J^1(\tau^\ast, \Lambda^\ast;p) \leq  \mathcal J^1(\tau, \Lambda^\ast;p)\\
\mathcal J^2(\tau^\ast, \Lambda^\ast;p) \geq  \mathcal J^2(\tau^\ast, \Lambda;p)\end{array}\right.
\end{equation}
for any pair $(\tau,\Lambda)\in\mathbb T\times \mathbb L$.
\end{definition}

\begin{remark}
\label{rem}
We have equipped the fraudster with the filtration $\F^{W,\theta}$ so that any fraudster strategy is on the form
$\Lambda=\Lambda^0 1_{\{\theta=0\}} + \Lambda^1 1_{\{\theta=1\}}$ for non-decreasing $\F^{W}$-adapted processes 
$\Lambda^0$ and $\Lambda^1$. However, since neither $\mathcal J^1$ nor $\mathcal J^2$ depends on $\Lambda^0$, but only on $\Lambda^1$, 
the specification of $\Lambda^0$ is superfluous. 
In fact, a game which is strategically equivalent to the one introduced is obtained if $\mathcal J^2$ is replaced with 
\[\hat{\mathcal J}^2(\tau,\Lambda;p):=\E \left [\left. \int _0^{\tau} e^{-rs}  d\Lambda_s \right\vert \theta=1\right ].\]
The functionals $\mathcal J^2$ and $\hat{\mathcal J}^2$ have the following interpretations. Imagine that before the game starts, i.e., at time 
$t=0-$, neither player knows $\theta$, and that the value of $\theta$ will be revealed to Player~2 at $t=0$.
$\mathcal J^2$ is then the expected payoff for Player~2 at time $0-$, whereas $\hat{\mathcal J}^2$ is the 
expected payoff at time 0 in case $\theta=1$. These games are referred to as the {\em ex ante} game and the 
{\em interim} version of the game, respectively (see \cite{AM,Ha} for classical theory of games under incomplete information).
Also note that $\mathcal J^2(\tau,\Lambda;p)=p\hat{\mathcal J}^2(\tau,\Lambda;p)$ and that, in the definition of a NE, the second inequality in \eqref{NEineq}
can be replaced by $\hat{\mathcal J}^2(\tau^\ast, \Lambda^\ast;p) \geq  \hat{\mathcal J}^2(\tau^\ast, \Lambda;p)$. 
\end{remark}

\begin{remark}
An essential feature of our game is that the account holder cannot observe the strategy
$\Lambda$ used by the fraudster, but can only observe the path of the aggregate process $X$. Our setup and terminology thus differs from that of, e.g., \cite{CR} and \cite{PTZ}, where a strategy is a map from the set of controls of one player to the set of controls of the other player.
%
\end{remark}

In Section~\ref{sec:pure} we determine a Nash equilibrium for the above game in case $M\leq \frac{\sqrt{\pi}}{2\sqrt{r}}$,
see Theorem~\ref{main}. Then, in Section~\ref{sec:mixed}, a Nash equilibrium is obtained also for $M> \frac{\sqrt{\pi}}{2\sqrt{r}}$ by 
allowing for randomized stopping strategies, see Theorem~\ref{msNE}.

\subsection{Literature review}
While the present problem belongs to a new class of stochastic games --- the novel feature being the presence of unknown competition, or a ``ghost'' (cf. the below) --- it does belong in a wider sense to the literature on combined stochastic control and stopping games and to the literature on stochastic games under asymmetric information. Without aspiring to completeness, we give a brief review of such related papers. 

The controller-and-stopper game was introduced as a zero-sum game in discrete time in \cite{MS}. 
In the seminal study \cite{KSu}, the zero-sum controller-and-stopper game is studied in a one-dimensional diffusion setting. 
In \cite{HRSZ} a zero-sum game between a stopper and a singular controller of a one-dimensional diffusion is studied and, depending on which of the controller and the stopper has the first-move advantage, two different solutions are obtained (cf. Remark \ref{precedence:rem} above). A similar game is considered in \cite{HY} in a model based on a spectrally one-sided L\'evy process. 
Further literature on zero-sum controller-and-stopper games include
\cite{
BHOT,
BH,
CCP,
KZ,
NZ},
%
%
whereas \cite{CS, KL} study non-zero-sum versions.
%

The first study of a stochastic differential game under asymmetric information is \cite{CR} (see also \cite{CR2}). In particular, in \cite{CR} a zero-sum stochastic differential game under asymmetric information where two players control a multi-dimensional diffusion is studied. It is shown that the game has a value, which is the solution in a dual sense to an associated 
Hamilton-Jacobi equation. 
%
In \cite{GR}, a continuous time zero-sum game where one player observes a Brownian motion, and one does not, is studied using an approach which relies on an approximating sequence of corresponding discrete time games.  
Numerical approximation for stochastic differential games with asymmetric information is studied in \cite{BFR} and \cite{G}. Another strand of this literature studies stopping games under asymmetric information, see, e.g., \cite{DEG,DMP,GG,G2,LM}.

A key feature of the problem studied in this paper is the presence of unknown competition in the sense that Player~2 --- 
the fraudster --- may be non-existent (inactive), effectively leaving the unknowing Player~1 --- the account holder --- playing a game against a ``ghost''. 
In contrast to \cite{CR}, where the players can observe the realization of the other player's control, it is essential in our set-up that the actions of the fraudster are not directly observable by the account holder (if they were, the account holder would detect the fraudster as soon as the control is non-zero). It appears that this is the first study of a stochastic game with unknown competition and a continuously controllable state process --- the only other paper to consider dynamic stochastic games under unknown competition is, according to our knowledge, \cite{DAE}, in which a Dynkin (stopping) game is studied and where the term ``ghost'' is introduced. 
In particular, in \cite{DAE} the effect of unknown competition in a Dynkin game is studied in a setting where each of the two players is uncertain whether the other player exists or not; using methods from filtering theory it is shown that a key feature of the equilibrium solution to that problem is that randomized stopping strategies should be used in such a way that the other player's adjusted belief process, i.e., the conditional probability of active competition, stays below a certain boundary. For related studies within the economics literature see \cite{HKL} and \cite{MM}, where auctions with unknown competition in a non-dynamic setting are studied. 

Admittingly, the current set-up models a rather stylized version of fraud detection, and naturally has its limitations from an applied perspective. For example, the set-up of a fraudster that is either active or inactive 
may seem unrealistically static, and one could allow for a more dynamic presence of, 
possibly, several fraudsters.
Also, the account holdings are assumed unlimited so that the analysis below becomes
independent of the present value of $X$, whereas limited resources would be a natural 
ingredient in many applications. Moreover, the set of possible actions (stop or not stop) of the account holder is rather scarce; applications in resource extraction would call for a continuous action space also for her.
It is our hope that the methods developed in the present paper will facilitate and encourage further research into dynamic stochastic games under unknown competition.

%
%
%

\section{Applications of filtering theory} \label{sec:filter}


The Nash equilibria we identify below are specified in terms of the conditional probability  the account holder assigns to the event $\{\theta=1\}$ that the fraudster is active, and we therefore require some elements of filtering theory. Moreover, in our Nash equilibria the fraud strategies $\Lambda^\ast$ are absolutely continuous in the sense that $\Lambda^\ast_t=\int_0^t\lambda^\ast_s\,ds$ for a positive process $\lambda^\ast$. 

Let us first view the situation from the perspective of the account holder. If the fraudster uses a fixed strategy 
$\Lambda^\ast\in\mathbb L$ on the form $\Lambda^\ast_t=\int_0^t\lambda^\ast_s\,ds$, then, under the assumption that $\lambda^\ast_t$ also satisfies certain integrability conditions, we obtain from filtering theory (see, e.g., \cite[Chapter~8]{LS}) that the 
\textit{innovations process} 
\[\hat W_t:=X_t + \int_0^t \E[\theta\lambda^\ast_s \vert \mathcal F^X_s] \,ds\]
is a Brownian motion with respect to $(\F^X,\P)$.
Moreover, 
\[\E[\theta\lambda^\ast_t \vert\mathcal F^X_t]=\P(\theta=1\vert \mathcal F^X_t)\E[\lambda^\ast_t \vert\mathcal F^X_t, \theta=1]
=P_t\hat{\lambda}_t,\]
where 
\[\hat\lambda_t:=\E[\lambda^\ast_t \vert\mathcal F^X_t, \theta=1]\]
and the process $P_t:=P^{\Lambda^\ast}_t:=\E[\theta \vert \,  \F^X_t]$ satisfies 
\begin{equation} \label{Peqn}
dP_t= -\hat\lambda_t P_t (1-P_t)\, d\hat W_t
\end{equation}
with $P_0=p$.
Note that if $\lambda^\ast$ is in particular $\F^X$-adapted (which is the case for the equilibria we identify in this paper, as we will see), then $\hat\lambda=\lambda^\ast$ so that
\begin{equation*} 
dP_t= -\lambda^\ast_t P_t (1-P_t) \,d\hat W_t.
\end{equation*}

The process $P=P^{\Lambda^\ast}$ given by \eqref{Peqn} is the conditional probability of the existence of the fraudster 
provided the fraudster uses the strategy $\Lambda^\ast_t=\int_0^t\lambda^\ast_s\,ds$. However, it is essential in our
problem set-up that the actual strategy is not directly observable, which means that we also need to take the effect of
possible deviations from a NE into account. To do that we note that if the fraudster deviates from $\Lambda^*$ and instead uses a strategy $\Lambda$, then the dynamics in \eqref{Peqn} can be expressed as
\begin{eqnarray} \label{P}
dP_t &=& - \hat\lambda_tP_t (1-P_t) (dX_t + \hat\lambda_t P_t \,dt)\\
\notag
&=& - \hat\lambda_t P_t (1-P_t) (-\theta\, d\Lambda_t + \hat\lambda_t P_t\, dt) - \hat\lambda_t P_t (1-P_t ) \,dW_t,
\end{eqnarray}
thus showing exactly how the fraudster may manipulate the belief of the account holder, i.e. the process $P$, by controlling the process $X$. 
Observe that the dynamics of $P$ given by \eqref{P} depends 
on the actual strategy $\Lambda$ used by the fraudster as well on the account holder's assumption of the fraud strategy $\Lambda^\ast_t=\int_0^t\lambda^\ast_s\,ds$. 
In particular, if we condition on 
$\{\theta=1\}$ (the fraudster is active) and suppose that the fraudster uses an arbitrary strategy $\Lambda$, possibly different from $\Lambda^\ast_t=\int_0^t\lambda^\ast_s\,ds$ as assumed by the account holder, then $P=P^\Lambda$ is given by 
\begin{equation} \label{eq:fraudsterP}
dP_t =  -\hat\lambda_t P_t (1-P_t) (-d\Lambda_t + \hat\lambda_t P_t \,dt) - \hat\lambda_t P_t (1-P_t )\, dW_t.
\end{equation}
Furthermore, if the fraudster uses the strategy $\Lambda^\ast$ as assumed by the account holder, and if $\lambda^\ast$ is $\F^X$-adapted so that $\hat\lambda=\lambda^\ast$, then the drift in \eqref{eq:fraudsterP} is $(\lambda^\ast_t)^2 P_t (1-P_t)^2$ which is positive, and in this case the fraudster thus gradually reveals her existence to the account holder. However, note that the drift can also be negative (at most $-(\lambda^\ast_t)^2 P^2_t (1-P_t)$, corresponding to a flat portion of $\Lambda$), or, at the other extreme, arbitrarily large (corresponding to a rapidly increasing $\Lambda$).


%
%

\section{A pure Nash equilibrium} \label{sec:pure}

We provide heuristic calculations in Section~\ref{sec:purestrategies} to obtain a candidate NE. The candidate NE is summarized in Section~\ref{summary}, and further properties are obtained in Section~\ref{properties}. In Section~\ref{sec:verification} we then verify that this candidate indeed constitutes
a NE.

\subsection{Deriving a candidate Nash equilibrium} \label{sec:purestrategies}
In this section we will look for a NE, cf. Definition~\ref{NE}. We emphasize that the calculations in this section are mainly motivational; a formal verification result is provided in Section~\ref{sec:verification} below.

It is reasonable to guess that if the account holder is sufficiently sure that the fraudster is active, she will pay the amount $M$ to deactivate the fraudster, and that, 
from the viewpoint of the fraudster, there should exist an optimal push rate, depending on the account holder's 
current belief, which solves the trade-off between stealing and avoiding detection.
In fact, we will look for a NE where the fraud strategy is of the form
\begin{equation}\label{Markovian-fraud-strat}
\Lambda^\ast_t=\int_0^t\lambda^\ast(P_s)\,ds
\end{equation}
for some non-negative function $\lambda^\ast$ to be determined (from now on $\lambda^\ast$ denotes such a function), 
and the stopping strategy takes the form of a threshold time 
\begin{equation*}
\tau^\ast=\tau^b:=\inf \{t\geq 0 : P_t \geq b\}
\end{equation*}
for some $b\in(0,1)$, where $P$ corresponds, in equilibrium, to the conditional probability for the account holder that the fraudster is active. 

More precisely, given $\Lambda\in\mathbb L$, consider a pair $(X,P)=(X^{\Lambda},P^{\Lambda})$ given by 
\begin{equation} 
\label{syst}
\left\{\begin{array}{ll}
dX_t=-\theta d\Lambda_t+dW_t\\
dP_t=- \lambda^\ast(P_t)P_t (1-P_t) (dX_t + \lambda^\ast(P_t)P_t \,dt) \end{array}\right.
\end{equation}
with $X_0=0$ and $P_0=p$.
Then the dynamics of $P$ conditioned on $\{\theta=1\}$ are 
\begin{align}\label{Sec:3.1:eqP-fraudster}
dP_t =  -\lambda^\ast(P_t)P_t (1-P_t) (-d\Lambda_t+ \lambda^\ast(P_t)P_t\, dt) - \lambda^\ast(P_t)P_t (1-P_t ) \,dW_t. 
\end{align}
Hence, if the account holder uses a threshold strategy $\tau^b$ for the process $P$, the active fraudster faces a stochastic control problem
\[v(p) =\sup_{\Lambda \in \mathbb L} \hat {\mathcal J}^2(\tau^b,\Lambda;p) = 
\sup_{\Lambda \in \mathbb L}  \E \left [ \left.\int _0^{\tau^b} e^{-rs} d\Lambda_s \right\vert\theta=1\right]
\]
with the underlying process being $P$ given by \eqref{syst} or, equivalently, 
\begin{equation}
\label{optimalproblem:fraudster}
v(p) = 
\sup_{\Lambda \in \mathbb L}  \E \left[\int _0^{\tau^b} e^{-rs} d\Lambda_s \right]
\end{equation}
with $P$ given by \eqref{Sec:3.1:eqP-fraudster} with $P_0=p$.

\begin{remark}
In the system \eqref{syst} we use the function $\lambda^\ast$ (to be determined) to specify the dynamics of $P$, rather than 
$\hat\lambda_t:=\E[\lambda^\ast(P_t)\vert\F^X_t]$ as in Section~\ref{sec:filter}. 
We will see below (cf. Proposition~\ref{thm:filter-etc})
that if the fraudster uses the control $\Lambda_t=\Lambda^\ast_t:=\int^t_0\lambda^\ast(P_s)\,ds$, then 
$\Lambda^\ast$ is $\F^X$-adapted; consequently, in that case, $\hat\lambda_t=\lambda^\ast(P_t)$,
and $P_t$ then coincides with $\E[\theta\vert\F^X_t]$. Hence $P$ corresponds, in equilibrium, to the conditional probability for the account holder that the fraudster is active. 
\end{remark}

Using the dynamic programming principle we expect that if $\lambda^\ast$ corresponds to a NE fraud strategy 
--- i.e. if $\Lambda^\ast_t=\int_0^t\lambda^\ast(P_s)\,ds$ with $P$ given by \eqref{Sec:3.1:eqP-fraudster} is optimal in \eqref{optimalproblem:fraudster} --- then it should hold that
\begin{equation}\label{eq:suboptimality}
\frac{1}{2}(\lambda^\ast)^2p^2(1-p)^2v_{pp}-(\lambda^\ast)^2p^2(1-p) v_p + \lambda^\ast p (1-p) \lambda  v_p -rv +\lambda  \leq 0 
\end{equation} 
for all constants $\lambda\geq 0$ in the continuation region of the NE stopping strategy, i.e. for $0<p<b$.  
(To ease the presentation we have suppressed the dependence on $p$ in the functions in \eqref{eq:suboptimality}, a convention we will use throughout the paper.) 
Similarly, by optimality of $\lambda^\ast$ we expect equality in 
\eqref{eq:suboptimality} when $\lambda=\lambda^\ast$, i.e.,
\begin{equation} \label{eq:optimality}
\frac{1}{2}(\lambda^\ast )^2p^2(1-p)^2 v_{pp}-(\lambda^\ast)^2p^2(1-p) v_p + (\lambda^\ast)^2 p (1-p) v_p  -rv +\lambda^\ast = 0
\end{equation}
for $0<p<b$.
Subtracting  \eqref{eq:suboptimality} from \eqref{eq:optimality} yields
\begin{align*}
(\lambda^\ast -\lambda) (\lambda^\ast p (1-p) v_p +1) \geq 0
\end{align*}
for all $\lambda\geq 0$, which implies that 
\begin{align*}
\lambda^\ast > 0
\end{align*}
and hence
$$
\lambda^\ast p (1-p) v_p +1  =0
$$
for $0<p<b$. Inserting
\begin{align*}
\lambda^\ast = -\frac{1}{p(1-p)v_p}
\end{align*}
into \eqref{eq:optimality} and multiplying by $v_p$ yields a non-linear ordinary differential equation
\begin{align} \label{eq:fraudsterode}
 \frac{v_{pp}}{2v_p}- \frac{1}{1-p} -rvv_p=0.
\end{align}
This can be integrated, and using separation of variables we find 
\begin{align*}
e^{-rv^2} \,dv = \frac{A}{(1-p)^2} \,dp.
\end{align*}
The general solution thus satisfies 
\begin{equation*}
\Phi (\sqrt{2r}v(p)) =\frac{B}{1-p} + C ,
\end{equation*}
where 
\[\Phi(y)=\int_{-\infty}^y\varphi(z)\,dz\quad \mbox{and}\quad \varphi(z)=\frac{1}{\sqrt{2\pi}}\exp\{-z^2/2\}\]
are the distribution function and the density of the standard normal distribution, respectively.

Imposing the boundary conditions $v(b)=0$ (which we expect since there is no possibility left for the fraudster to steal after $\tau^b$) and $v(0+):=\lim_{p \searrow 0}v(p) = \infty$ (corresponding to unbounded possibilities for the fraudster to steal in the limit as 
$p\to 0$) yields 
%
\begin{equation} \label{valueu}
\Phi(\sqrt{2r}v(p)) =\frac{1}{2b} + \frac{1}{2} - \frac{1}{2b} \frac{1-b}{1-p}.
\end{equation}
The function $v(p)$ implicitly given in \eqref{valueu} is thus a candidate equilibrium value function for the fraudster 
(recall, however, that the threshold $b$ is yet to be determined).

We now turn to the perspective of the account holder. For a given fraud strategy $\Lambda^\ast$ of the form \eqref{Markovian-fraud-strat}, the account holder faces an optimal stopping problem 
\begin{align}\label{optimalproblem:accountholder}
u(p):=\inf_{\tau \in \mathbb T}{\mathcal J}^1(\tau,\Lambda^\ast;p) = 
\inf_{\tau \in \mathbb T}\E\left [ \theta\int _0^{\tau} e^{-rs} \lambda^\ast(P_s)\,ds   + e^{-r\tau}M\I{\tau<\infty} \right ],
\end{align}
where we expect that the underlying process $P=P^{\Lambda^*}$ given by \eqref{syst} (with $\Lambda=\Lambda^*$) coincides with $\E[\theta \vert  \F^X_t]$, and that 
$$
dP_t = -\lambda^\ast(P_t) P_t (1-P_t) \,d\hat W_t,
$$
where $\hat W:= X_t + \int_0^t\lambda^\ast(P_s)P_s\,ds$ is a Brownian motion with respect to $(\F^X,\P)$, see Section \ref{sec:filter}. Arguing heuristically, we may thus use iterated expectations to replace $\theta$ with $P$ in \eqref{optimalproblem:accountholder} and then we expect, based on the dynamic programming principle, that 
\begin{align} \label{eq:accountode}
\frac{(\lambda^\ast)^2 p^2 (1-p)^2}{2}  u_{pp}-r u +p\lambda^\ast  &=0, \quad 0<p<b \\
u(0) &= 0 \notag\\ 
u(b)& =M, \notag
\end{align}
where the boundary conditions come from the account holder having zero expected cost if there is no fraudster and from immediate stopping at the boundary $b$. 

Recall from the discussion above that a NE fraud strategy should satisfy 
\begin{align*} 
\lambda^\ast=  -\frac{1}{p(1-p)v_p}  = \sqrt{2r} \varphi(\sqrt{2r}v) \frac{2b}{1-b} \frac{1-p}{p} 
\end{align*}
(the second equality is derived from \eqref{valueu}),
which inserted into \eqref{eq:accountode} gives
\begin{equation} \label{eq:accountode2}
\frac{u_{pp}}{2 v_p}   -r uv_p- \frac{1}{1-p} =0
\end{equation}
for $0<p<b$.
Comparing with \eqref{eq:fraudsterode} we see that $u=v$ is a particular solution to \eqref{eq:accountode2}.
Moreover, making the Ansatz $u(p) = (1-p) f(v(p))$ for the homogenous part (for some function $f$ to be determined) yields
\begin{align*}
u_p&= (1-p) f'(v) v_p  - f(v) \\
u_{pp} &= (1-p) f''(v) v_p^2 + (1-p) f'(v) v_{pp} -2 f'(v) v_p.
\end{align*}
Inserting these expressions into \eqref{eq:accountode2} and using that $\frac{v_{pp}}{2v_p}= rvv_p + \frac{1}{1-p}$ by \eqref{eq:fraudsterode} we find that the homogeneous part of \eqref{eq:accountode2} can be written as
$$
\frac{1}{2}f''(v) + r v f'(v) - rf(v) = 0.
$$
The linear ordinary differential equation 
$$\frac{1}{2}f''(x) + r x f'(x) - r f(x) =0 $$
has general solution 
\begin{equation*}
Ax +B F(\sqrt{2r}x),
\end{equation*}
where
\[F(y):=\varphi(y)-y \Psi(y)\]
and
\[\Psi(y):=1-\Phi(y)=\int_y^\infty\varphi(z)\,dz.\]
A candidate for the equilibrium value function of the account holder is thus
\begin{equation} \label{eq:valueu}
u(p) = A (1-p)  v(p) + B (1-p)  F(\sqrt{2r}v(p)) + v(p),
\end{equation}
where $A$ and $B$ (and $b$ which is hidden implicitly in $v$) are yet to be determined. 

Before turning to these constants, we make some observations concerning the function $F$. 

\begin{lemma}  \label{lemF}
The function $F$ satisfies 
\begin{itemize}
\item[(i)]
$ F(y) \geq 0$;
\item[(ii)]
$\lim_{y\to \infty} F(y) =0$;
\item[(iii)]
$\lim_{y \searrow 0} F(y)=\frac{1}{\sqrt{2\pi}}$;
\item[(iv)]
$\lim_{y \searrow 0} F'(y) = -\frac{1}{2}$.
\end{itemize}
\end{lemma}

\begin{proof}
(i) is immediate from the standard estimate
\[y\int_y^\infty\varphi(z)\, dz\leq \int_y^\infty z \varphi(z)\, dz=\varphi(y).\]
(ii) then follows since (i) implies $0\leq F(y)\leq \varphi(y)\to 0$ as $y\to\infty$.
(iii) is obvious.
Finally, (iv) follows from $F'(y)=-\Psi(y)$.
\end{proof}

We now impose boundary conditions for the candidate equilibrium value function $u(p)$. More precisely, we wish to determine the constants $A, B$, and $b$ so that, first of all, $u(0+)=0$ and $u(b)=M$.
Recalling $v(0+)=\infty$ and $v(b)=0$, and using (ii) and (iii) from Lemma~\ref{lemF}, we get from \eqref{eq:valueu} that 
$$
 A=-1\quad \quad \mbox{and} \quad \quad B= \frac{M \sqrt{2\pi}}{1-b},
$$
so 
\[u(p) = p  v(p) + M \sqrt{2\pi}\frac{1-p}{1-b}   F(\sqrt{2r}v(p)) .\]
To determine the unknown stopping boundary $b$, we
impose the principle of smooth fit from optimal stopping theory, which for us takes the form $u_p(b)=0$. We have that
\[
u_p(p)= v(p) +p v_p(p) -\frac{M \sqrt{2\pi}}{1-b} \left (F(\sqrt{2r}v(p)) + (1-p)\Psi(\sqrt{2r}v(p)) \sqrt{2r}v_p(p)   \right) \]
so
\begin{eqnarray}\label{exp}
u_p(b) &=&  b v_p(b)  - \frac{M \sqrt{2\pi}}{1-b} \left(F(\sqrt{2r}v(b)) + (1-b)\Psi(\sqrt{2r}v(b)) \sqrt{2r}v_p(b)      \right)  \\
\notag
&=&  b v_p(b) - \frac{M}{1-b} - M  \sqrt{r\pi}  v_p(b)  .
\end{eqnarray}
Differentiation of \eqref{valueu} gives 
\begin{equation}\label{vp}
v_p(b) = -\frac{ \sqrt{\pi}}{2\sqrt{r}b(1-b)},
\end{equation}
and plugging this and $u_p(b)=0$ into \eqref{exp} yields
\begin{equation*}
 b= \frac{M \pi\sqrt{r}}{\sqrt{\pi} +2M\sqrt{r}}.
\end{equation*}

\subsection{The candidate Nash equilibrium}\label{summary}

We now summarize the specification of our candidate Nash equilibrium $(\tau^b,\Lambda^\ast)$. 
Let
\begin{equation}
\label{b}
b:= \frac{M \pi\sqrt{r}}{\sqrt{\pi} +2M\sqrt{r}},
\end{equation}
and assume that 
\begin{equation}\label{hatM}
M\leq \hat M:=\frac{\sqrt\pi}{2\sqrt r}
\end{equation} 
(this bound has not been discussed yet, but we include it here as it is essential in the verification below, cf. Lemma~\ref{ubound} below).
The expression \eqref{b} for $b$ is increasing in $M$, and thus
\begin{align*}
b= \frac{M \pi\sqrt{r}}{\sqrt{\pi} +2M\sqrt{r}}\leq  \frac{\hat M \pi\sqrt{r}}{\sqrt{\pi} +2\hat M\sqrt{r}}=\pi/4<1.
\end{align*}
Define $v$ (the candidate equilibrium value function for the fraudster in the interim version of the game, cf. Remark \ref{rem}) by
\begin{equation}
\label{eq:value1}
 v(p) =\begin{cases} \frac{1}{\sqrt{2r}}\Phi^{-1}\left(1- \frac{(1-b)p}{b(1-p)} \frac{1}{2}\right) & 0<p<b \\ 0 & p\geq b \end{cases}
\end{equation}
and $u$ (the candidate equilibrium value function for the account holder) by 
\begin{equation}  
u(p) = \begin{cases} 
p  v(p) + M \sqrt{2\pi}\frac{1-p}{1-b}   F(\sqrt{2r}v(p)) & 0<p<b \\ M & p\geq b.
\end{cases}
\label{eq:value2}
\end{equation}
Next, let
\begin{equation}
\lambda^\ast(p)= \left\{\begin{array}{ll}
\frac{-1}{p(1-p)v_p(p)} & 0<p< b \\
\frac{2b\sqrt{r}}{p\sqrt{\pi}}  & p\geq b
 \end{array}\right.
 \label{eq:strategy2}
\end{equation}
(using \eqref{vp} one sees that the second line of \eqref{eq:strategy2} corresponds to a continuous extension of $\lambda^\ast$). Furthermore, let $(X,P)=(X^{\Lambda},P^{\Lambda})$ for $\Lambda\in \mathbb L$ be defined by
\begin{equation}\label{eq:P-equation}
\left\{\begin{array}{ll}
dX_t=-\theta \, d\Lambda_t + dW_t\\
dP_t=- \lambda^\ast(P_t)P_t (1-P_t) (dX_t + \lambda^\ast(P_t) P_t \,dt) \end{array}\right.
\end{equation}
with $X_0=0$ and $P_0=p$.
Our candidate NE is given by $(\tau^b,\Lambda^\ast)$ with 
\begin{equation} \label{eq:strategy1}
\tau^b = \inf	\{t\geq 0: P_t \geq b\}
\end{equation}
and
\begin{equation}\label{eq:strategy2K}
\Lambda^\ast_t= \int_0^t\lambda^\ast(P_s)\,ds.
\end{equation}

For graphical illustrations of $\lambda^\ast$, $u$ and $v$, see Figures~\ref{fig:lambda}-\ref{fig:uv}.

\begin{figure}[h]
  \subfigure[]{\includegraphics[width=0.45\textwidth]{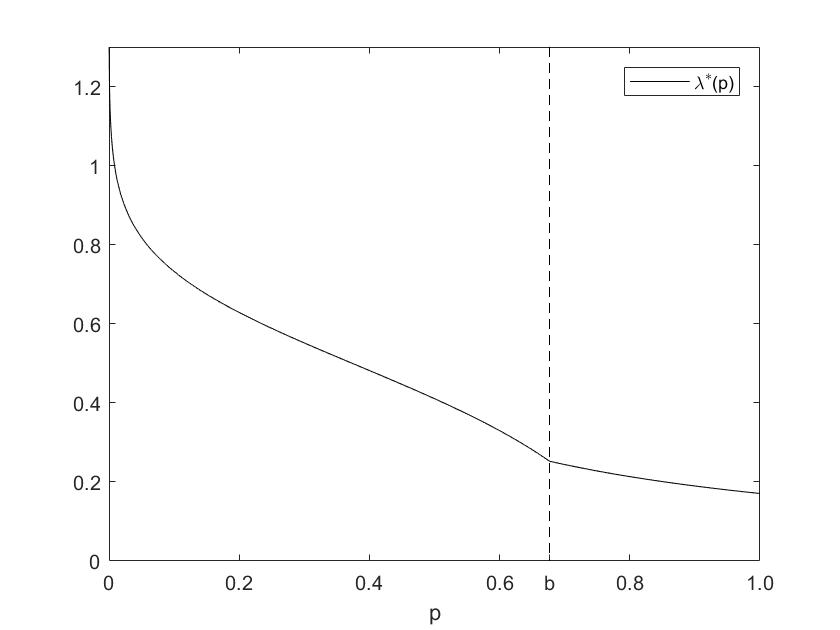}\label{fig:lambda}}
  \subfigure[]{\includegraphics[width=0.45\textwidth]{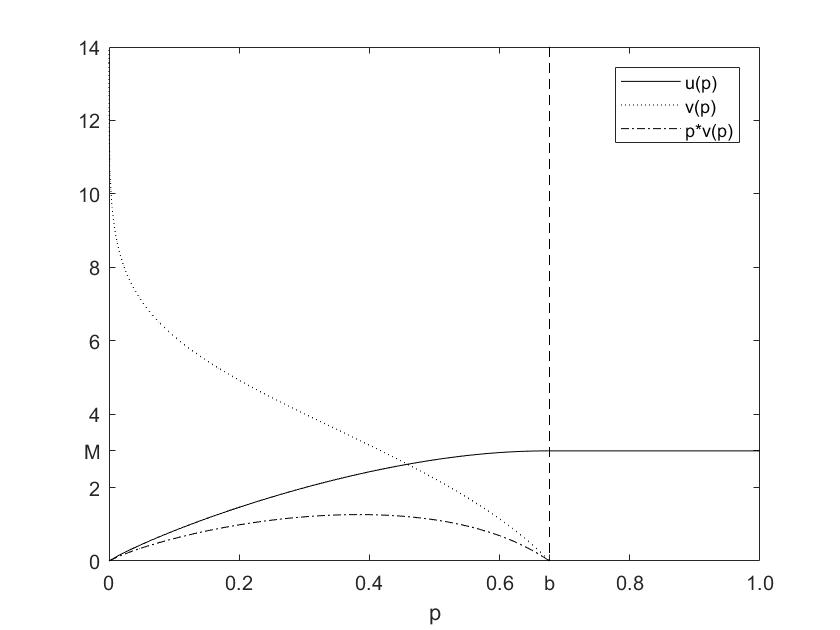}\label{fig:uv}}
\caption{
The intensity of the equilibrium fraudster strategy $\lambda^*$,   
the threshold for the account holder equilibrium stopping strategy $b$, and   
the equilibrium value functions $u$ (account holder) and $v$ (fraudster, interim version of the game); cf. Section \ref{summary}. We also include the fraudster equilibrium value function $pv$ in the ex ante game, cf. Remark~\ref{rem}. The parameters
are $r=0.05$ and $M=3$ (note that $M<\hat M\approx 3.96$, where $\hat M$ is defined in \eqref{hatM}).}
\end{figure}

\begin{remark} \label{existence-strong-sol-remark}
Suppose a fraud strategy $\Lambda \in \mathbb L$ is given. Then $X$ is given by $X_t=W_t$ on $\{\theta=0\}$ and by $X_t=W_t+\Lambda_t$ on $\{\theta=1\}$. Thus the second equation of \eqref{eq:P-equation} is a stochastic differential equation driven by a
continuous semimartingale $X$ with drift and diffusion coefficients 
\begin{equation} \label{drift-coeff1}
- (\lambda^\ast)^2p^2(1-p) =
\left\{
\begin{array}{ll}
\frac{-1}{(1-p)v_p^2} & 0<p<b\\
-\frac{4r}{\pi}b^2(1-p)  & p \geq b 
\end{array}
\right.
\end{equation}
and
\begin{equation} \label{diffusion-coeff}
-\lambda^\ast p(1-p) =
\left\{
\begin{array}{ll}
\frac{1}{v_p} & 0<p<b\\
-\frac{2\sqrt{r}b}{\sqrt{\pi}}(1-p) & p \geq b,
\end{array}\right.
\end{equation}
respectively.
These coefficients are Lipschitz on any interval $[\epsilon,1]$, $\epsilon>0$, so a strong solution $P$ exists, at least up to a potential explosion at $0$. 

Also note that the fraudster observes $\theta$ and chooses $\Lambda$, and can then solve the system 
\eqref{eq:P-equation}. The account holder, on the other hand, only observes a path of $X$, and calculates $P$ using the second equation of 
\eqref{eq:P-equation} in a pathwise manner. By the results of \cite{Ka}, they obtain the same $P$.
\end{remark}

To show that the process $P$ cannot reach zero in finite time we will use the following lemma.

\begin{lemma} \label{lem:fellertest} 
There exist constants $A>0$ and $B>0$ such that, for all $x \geq B$, it holds that
\begin{align*}
\varphi(x) \leq A \sqrt{\ln\left( \frac{1}{\Psi(x)}\right)} \Psi(x).
\end{align*}
\end{lemma}

\begin{proof}
Recall that $x\Psi(x) \leq \varphi(x)$ so that $\frac{1}{\Psi(x)}\geq \frac{x}{\varphi(x)}$, and that $\lim_{x\rightarrow\infty}\frac{\varphi(x)}{x\Psi(x)}=1$.  Thus there exist constants $B,C,D>0$ such that for all $x\geq B$ it holds that
\begin{eqnarray*}
C \frac{x\Psi(x)}{\varphi(x)}\sqrt{\ln\frac{1}{\Psi(x)}} 
& \geq& \sqrt{\ln\frac{1}{\Psi(x)}}
 \geq \sqrt{\ln x-\ln \varphi(x)}\\
& \geq& \sqrt{\ln x+\frac{x^2}{2}} \geq D x.
\end{eqnarray*}
The result thus follows by setting $A=C/D$.
\end{proof}

\begin{proposition} \label{thm:existence-etc} 
Given any fraud strategy $\Lambda \in \mathbb L$, the SDE \eqref{eq:P-equation} has a strong solution $(X,P)=(X^{\Lambda},P^{\Lambda})$, with
\begin{align} 
\tau_0:=\inf\{t\geq 0: P_t=0\}=\infty \enskip a.s.\label{stopP-inf}
\end{align}
Moreover, the fraud strategy $\Lambda^\ast$ in \eqref{eq:strategy2K} is admissible in the sense that it satisfies 
$\Lambda^\ast \in \mathbb L$.   
\end{proposition}

\begin{proof} 
Suppose a fraud strategy $\Lambda \in \mathbb L$ is given. The existence of a strong solution was established in Remark \ref{existence-strong-sol-remark}. 
To see that \eqref{stopP-inf} holds, note that it suffices to show that \eqref{stopP-inf} holds on $\{\theta=0\}$, since it then
follows from comparison (see \cite[Chapter IX.3]{RY}) that \eqref{stopP-inf} holds also on $\{\theta=1\}$ since $\Lambda$ is non-decreasing (implying that the drift for $P$ is minimal when $\theta=0$). Again by comparison, it suffices to check non-explosion at $0$ for the SDE 
\[
d\tilde P_t = - \frac{C}{v_p^2(\tilde P_t)} - \frac{1}{v_p(\tilde P_t)}\,dW_t
\]
for any constant $C>0$. For such a diffusion, the densities of the scale function and speed measure are
\[s'(a) = e^{2Ca}\]
and 
\[m(p)= 2v_p^2(p)e^{-2Cp},\]
respectively. Using Fubini's Theorem, we thus obtain 
\begin{eqnarray*}
\int_0^z\left(\int_a^zm(p)dp\right)s'(a)\,da 
& =& \int_0^z  \left(\int_0^p e^{2Ca}da\right) 2v_p^2(p)e^{-2Cp}\, dp \\
& \geq& D_1\int_0^z  p v_p^2(p) \, dp
\end{eqnarray*}
(here and below $D_i$, $i=1,2,3$ denote positive constants).
From \eqref{eq:value1} we have
\begin{equation}
\label{psiprop}
\Psi(\sqrt{2r}v(p)) = 1-\Phi(\sqrt{2r}v(p))=\frac{(1-b)p}{2b(1-p)},
\end{equation}
which differentiated yields
\[ pv_p^2(p) = \frac{D_2p}{(1-p)^4\varphi^2(\sqrt{2r}v(p))}\geq \frac{D_2p}{\varphi^2(\sqrt{2r}v(p))}.\]
Using the above and Lemma \ref{lem:fellertest} (assuming that $z$ is small enough) we find that
\begin{eqnarray*}
\int_0^z\left(\int_a^zm(p)dp\right)s'(a)\,da 
& \geq&  D_1D_2\int_0^z \frac{p}{\varphi^2(\sqrt{2r}v(p))}\,dp\\
& \geq&  A^{-2}D_1D_2 \int_0^z \frac{p}{ \ln\left(\frac{1}{\Psi(\sqrt{2r}v(p))}\right) \Psi^2(\sqrt{2r}v(p))}\,dp\\
& =&  A^{-2}D_1D_2\int_0^z \frac{p}{ \left(\frac{1-b}{2b}\frac{p}{1-p}\right)^2\ln\left(\frac{2b}{1-b} \frac{1-p}{p}\right) }\,dp\\
& \geq&  D_3\int_0^z \frac{1}{ -p\ln p}\,dp = \infty.
\end{eqnarray*}
Consequently, \eqref{stopP-inf} follows by Feller's test for explosion.

Now consider  $\Lambda^\ast$ defined in \eqref{eq:strategy2K}. Then the dynamics for $P$ in \eqref{eq:P-equation} can be written as
\begin{align}
dP_t 
=   (\lambda^\ast(P_t))^2P_t(1-P_t)(\theta-P_t)\,dt  - \lambda^\ast(P_t)P_t (1-P_t) \,dW_t.\label{appPver}
\end{align}
Note that the diffusion coefficient in \eqref{appPver} is also in this case given by \eqref{diffusion-coeff}, and that the drift coefficient is given by \eqref{drift-coeff1} on $\{\theta=0\}$, and by
\begin{equation} \label{drift-coeff2}
(\lambda^\ast(P_t))^2P_t(1-P_t)^2 =
\left\{
\begin{array}{ll}
\frac{1}{pv_p^2(p)} & 0<p<b \\
\frac{4r}{\pi}\frac{b^2}{p}(1-p)^2 & p\geq b.
\end{array}\right.
\end{equation}
on $\{\theta=1\}$. However, the drift coefficient \eqref{drift-coeff2} is also locally Lipschitz and we thus obtain the existence of a strong solution $P$ as in the case above. Using also that $\lambda^\ast$ in \eqref{eq:strategy2} is a positive function (which is easy to verify) we conclude that 
$\Lambda^\ast_t=\int_0^t\lambda^\ast(P_s)\,ds$ is a continuous non-decreasing process adapted to $\F^{W,\theta}$, so $\Lambda^\ast \in \mathbb L$.  
\end{proof}

\subsection{Properties of the candidate Nash equilibrium}
\label{properties}
In this section we derive a few further properties of the candidate Nash equilibrium that are needed in the verification below.

\begin{lemma}\label{ubound}
Assume that $M\leq \hat M$.
Then $u$ defined in \eqref{eq:value2} 
(with $v$   in \eqref{eq:value1}) satisfies $u\leq M$.
\end{lemma}

\begin{proof}
To prove the claim we will show that the function $u$ is concave provided
$M\leq \hat M$. Indeed, concavity of $u$ together with the smooth fit condition $u_p(b)=0$ imply $u\leq M$.
For the concavity of $u$ it follows from \eqref{eq:accountode} that it suffices to prove $ru\leq \lambda^\ast p$ for all $p\in(0,b)$.

To do that, note that 
\begin{eqnarray*}
u(p) &=& p  v(p) + M \sqrt{2\pi}\frac{1-p}{1-b}   F(\sqrt{2r}v(p))\\
&=& p v(p) + M \sqrt{2\pi}\frac{1-p}{1-b} \varphi(\sqrt{2r}v(p)) -  M \sqrt{r\pi} v(p)\frac{p}{b},
\end{eqnarray*}
where we used \eqref{psiprop}.
Consequently,
\begin{eqnarray*}
\lambda^\ast p-ru &=& \sqrt{2r} \varphi(\sqrt{2r}v(p)) \frac{2b(1-p)}{1-b}  -r p  v(p) -M r\sqrt{2\pi}\frac{1-p}{1-b} \varphi(\sqrt{2r}v(p))
+rM \sqrt{r\pi} v(p)\frac{p}{b}\\
&=& \left(2b    - M \sqrt{r\pi}   \right)\sqrt{2r} \frac{1-p}{1-b} \varphi(\sqrt{2r}v(p)) + \left( M  \sqrt{r\pi}  -b\right) \frac{r p v(p)}{b}.
\end{eqnarray*}
Since $2b \geq M \sqrt{r\pi}$ follows from $M\leq\hat M$, and since 
\[M  \sqrt{r\pi}  -b=M\sqrt{r\pi}\left(1-  \frac{ \sqrt{\pi}}{\sqrt{\pi} +2M\sqrt{r}}\right)\geq 0,\]
it follows that $\lambda^\ast p\geq ru$, which completes the proof.
\end{proof}

We next establish that if the fraudster uses the candidate NE strategy then the process $P$ corresponds to the account holder's conditional probability of existence of the fraudster (in line with Section \ref{sec:filter}). 

\begin{proposition} \label{thm:filter-etc}
Suppose $\Lambda=\Lambda^\ast$ (see \eqref{eq:strategy2} and \eqref{eq:strategy2K}) in the SDE \eqref{eq:P-equation}, i.e., suppose the fraudster uses the (candidate) NE strategy. Then, for the solution $(X,P)=(X^{\Lambda^*},P^{\Lambda^*})$, it holds that 
\begin{align*}
P_t & =\E\left[\theta \vert \,  \F^X_t\right]
\end{align*}
and 
\begin{align} \label{thm:filter-etc:P}
dP_t & = -\lambda^\ast(P_t) P_t (1-P_t) \, d\hat W_t
\end{align}
a.s. with $P_0=p$, where 
$$\hat W_t:= X_t + \int_0^t\lambda^\ast(P_s)P_s\,ds$$
is a Brownian motion with respect to $(\F^X,\P)$. 
\end{proposition} 

\begin{proof} 
First observe that $P$ is $\F^X$-adapted, so that $\E[\theta\lambda^\ast(P_t)\vert\mathcal F^X_t]=\lambda^\ast(P_t)\Pi_t$, where
$\Pi_t:=\E[\theta\vert\mathcal F^X_t]$. Consequently, by \cite[Chapter 8.1]{LS} we have 
\[d\Pi_t=-\lambda^\ast(P_t)\Pi_t(1-\Pi_t)\,d\bar W_t,\]
where the innovations process
\[\bar W_t:=X_t+ \int_0^t \lambda^\ast(P_s)\Pi_s\,ds\]
is a Brownian motion with respect to $(\F^X,\P)$
(formally, to use \cite[Theorem 8.1]{LS} one needs to localize by, e.g., setting $\lambda^\ast_n:=\lambda^\ast\wedge n$; 
however, since $P$ does not reach 0 in finite time by Proposition~\ref{thm:existence-etc}, there is no problem when letting $n\to\infty$).
Thus 
\[d\Pi_t=-\lambda^\ast(P_t)\Pi_t(1-\Pi_t) (dX_t+ \lambda^\ast(P_t)\Pi_t\,dt),\]
and by \eqref{eq:P-equation},
\[dP_t=- \lambda^\ast(P_t)P_t (1-P_t) (dX_t + \lambda^\ast(P_t) P_t \,dt),\]
so by uniqueness of solutions we find that $P = \Pi$. Therefore, $\bar W=\hat W$, and \eqref{thm:filter-etc:P} holds.
\end{proof}

\subsection{A verification theorem for the pure Nash equilibrium}
\label{sec:verification}

In this subsection we verify that the pair $(\tau^b,\Lambda^\ast)$ defined above indeed constitutes a NE for our game, provided
$M\leq \hat M$ where $\hat M$ is defined in \eqref{hatM}.

\begin{theorem}
\label{main}
{\bf (Verification.)}
Assume that the account holder's cost for stopping satisfies $M\leq  \hat M$. Then the pair $(\tau^b, \Lambda^\ast)$ defined in \eqref{eq:strategy1} and \eqref{eq:strategy2K} is a NE in the sense of Definition~\ref{NE}. 
Moreover, the corresponding equilibrium value functions for the account holder and the fraudster (in the interim version of the game) are given by \eqref{eq:value2} and \eqref{eq:value1}, i.e.,
\[u(p)= \mathcal {\mathcal J}^1(\tau^b, \Lambda^\ast;p) \mbox{ and } v(p) = \hat{\mathcal J}^2(\tau^b, \Lambda^\ast;p).\]
\end{theorem}
\begin{proof}
{\bf Optimality of $\tau^b$.}
We first show that $\tau^b$ is an optimal stopping time for the account holder provided 
the fraudster uses $\Lambda^\ast$, i.e., that ${\mathcal J}^1(\tau, \Lambda^\ast;p) \geq {\mathcal J}^1(\tau^b,\Lambda^\ast;p)$ for any stopping time $\tau\in\mathbb T$.

By construction, $u \in C^{2}((0,b) \cap (b,1)) \cap C^1((0,1))$, and the limits $u_{pp}(b-)$ and $u_{pp}(b+)$ both exist, 
which is sufficient to apply It\^{o}'s formula in its standard form, cf. \cite[Problem 3.6.24]{KS}. 
Moreover, 
\begin{equation}
\label{ineq}
\frac{(\lambda^\ast)^2 p^2(1-p)^2}{2} u_{pp} -r u +\lambda^\ast p\geq 0
\end{equation}
for all $p\in(0,b)\cup(b,1)$. Indeed, for $p<b$, \eqref{ineq} holds with equality by \eqref{eq:accountode2}, and for 
$p>b$ we have $u=M$ so
\begin{equation}
\label{ineqb}
\frac{(\lambda^\ast)^2 p^2(1-p)^2}{2} u_{pp} -r u +\lambda^\ast p = -rM + \frac{2b\sqrt{r}}{\sqrt{\pi}}= \frac{rM(\sqrt\pi-2\sqrt rM)}{\sqrt\pi + 2M\sqrt{r}}\geq 0,
\end{equation}
where we used $b=\frac{M \pi\sqrt{r}}{\sqrt{\pi} +2M\sqrt{r}}$ and 
$M\leq \hat M=\frac{\sqrt{\pi}}{2\sqrt{r}}$.

Hence, using that the dynamics of $P$ satisfies \eqref{thm:filter-etc:P} and It\^{o}'s formula, we obtain for any stopping time $\tau\in\mathbb T$
\begin{eqnarray} \label{eq:expectationstopper}
e^{-r(\tau\wedge n)} u(P_{\tau\wedge n})  
&=& u(p) \notag\\
&& + \int_0^{\tau\wedge n} e^{-rt} 
\left ( \frac{(\lambda^\ast(P_t))^2 P_t^2(1-P_t)^2}{2} u_{pp} (P_t)-r u(P_t) \right)1_{\{P_t\not=b \}} dt \notag\\
&& - \int_0 ^{\tau \wedge n} e^{-rt} \lambda^\ast(P_t) P_t (1-P_t) u_p(P_t)d \hat W_t  \notag \\ 
\notag
&\geq&u(p)  - \int _0^{\tau\wedge n} e^{-rt} \lambda^\ast(P_t) P_t dt \\
&&
+ \int _0^{\tau \wedge n} e^{-rt}   \frac{u_p(P_t)}{v_p(P_t)} 1_{\{P_t<b\}} d \hat W_t.
\end{eqnarray}
The stochastic integral on the right hand side of \eqref{eq:expectationstopper} is in fact a martingale term since the integrand is bounded. Therefore, taking expectation in \eqref{eq:expectationstopper}, and using also Lemma~\ref{ubound}, Proposition~\ref{thm:filter-etc}, Fubini's theorem and iterated expectations, yields
\begin{eqnarray*}
u(p) & \leq& \E \left[e^{-r(\tau\wedge  n)} u(P_{\tau \wedge n}) + \int_0^{\tau\wedge n}e^{-rt}\lambda^\ast(P_t) P_t  \,dt\right] \\
& \leq&  \E \left[e^{-r(\tau \wedge n)} M + \int_0^{\tau \wedge n}e^{-rt}\lambda^\ast(P_t) 
\E [\theta | \,  \F^X_t]  \,dt\right]\\
& =& \E \left[e^{-r(\tau \wedge n)} M + \theta\int_0^{\tau \wedge n}e^{-rt}\lambda^\ast(P_t) \, dt\right].
\end{eqnarray*}
Using monotone convergence we thus find that
\[u(p)\leq  \E \left[e^{-r\tau} M \I{\tau <\infty} + \theta\int_0^{\tau}e^{-rs}\lambda^\ast(P_t)dt\right] = \mathcal {\mathcal J}^1(\tau, \Lambda^\ast; p).\]
Repeating the above argument with $\tau^b$ instead of $\tau$ gives 
\begin{eqnarray*}
u(p) &=& \E \left[e^{-r(\tau^b\wedge  n)} u(P_{\tau^b\wedge n}) + \theta\int_0^{\tau^b\wedge n}e^{-rt}\lambda^\ast(P_t)  \, dt\right]\\
&\geq&  \E \left[e^{-r\tau^b} M\I{\tau^b\leq  n} + \theta\int_0^{\tau^b \wedge n}e^{-rt}\lambda^\ast(P_t) \, dt\right],
\end{eqnarray*}
and monotone convergence thus yields 
\[u(p) \geq  \E \left[e^{-r\tau^b} M\I{\tau^b<\infty} + \theta \int_0^{\tau^b}e^{-rt}\lambda^\ast(P_t) dt\right]= \mathcal J^1(\tau^b, \Lambda^\ast; p). \]
Consequently, 
$$
u(p) =\mathcal  J^1(\tau^b, \Lambda^\ast; p) =\inf_{\tau\in\mathbb T}\mathcal J^1(\tau, \Lambda^\ast; p).
$$

{\bf Optimality of $\Lambda^\ast$.}
We now fix $\tau^b$ as defined in \eqref{eq:strategy1} and show that the strategy defined by \eqref{eq:strategy2K}
is an optimal strategy for the fraudster, 
i.e., that $\hat{\mathcal J}^2(\tau^b, \Lambda^\ast;p) \geq \hat{\mathcal J}^2(\tau^b,\Lambda;p)$ for any fraud strategy $\Lambda \in\mathbb L$. 
Note that the dynamics of $P$ in \eqref{eq:P-equation} can when conditioning on  $\{\theta=1\}$ be written as
\begin{align}\label{eq:P-equation-theta-equals1}
dP_t =  \lambda^\ast(P_t)P_t (1-P_t) (d\Lambda_t - \lambda^\ast(P_t)P_t \,dt) - \lambda^\ast(P_t)P_t (1-P_t )\,dW_t.
\end{align}
Let $\Lambda\in\mathbb L$ be an arbitrary fraud strategy, and consider the process 
$$N_t:=e^{-r(t\wedge\tau^b)}v(P_{t\wedge\tau^b})+ \int^{t\wedge\tau^b}_0e^{-rs}d\Lambda_s.$$
Using It\^{o}'s formula, \eqref{eq:fraudsterode} and \eqref{eq:strategy2}, we find that
\begin{eqnarray*}
dN_t &=& e^{-rt} \left ( \frac{1}{2}(\lambda^\ast(P_t))^2P_t^2 (1-P_t)^2 v_{pp}(P_t)  - rv(P_t) \right )\, dt  \\
&&+  e^{-rt}  \lambda^\ast(P_t) P_t (1-P_t) v_p(P_t)(d\Lambda_t - \lambda^\ast(P_t) P_t\,dt) +e^{-rt}d\Lambda_t\\
&&- e^{-rt}  \lambda^\ast(P_t) P_t (1-P_t) v_p(P_t)  \, dW_t \\
&= & e^{-rt}\,dW_t
\end{eqnarray*}
for $t\leq \tau^b$. Consequently, $N$ is a lower bounded martingale with bounded quadratic variation. 
The optional sampling theorem thus implies that  
\[v(p)=N_0= \E[N_{\tau^b}|\theta=1]\geq \E\left[\left.\int^{\tau^b}_0e^{-rt}d\Lambda_t\right\vert\theta=1\right]=\hat{\mathcal J}^2(\tau^b, \Lambda; p),\]
so $v(p) \geq \sup_{\Lambda \in \mathbb L} \hat{\mathcal J}^2(\tau^b, \Lambda; p)$ since $\Lambda$ was arbitrary. 

Conversely, if $\Lambda=\Lambda^\ast$, then \eqref{eq:P-equation-theta-equals1} reduces to 
\begin{align*} 
dP_t=  (\lambda^\ast(P_t))^2P_t (1-P_t)^2 dt- \lambda^\ast(P_t) P_t (1-P_t ) \,dW_t.
\end{align*}
We now claim that $\tau^b<\infty$ a.s. 
To see this, first note that the process $P$ cannot reach $0$ in finite time by Proposition~\ref{thm:existence-etc}. 
Moreover, the time-changed process 
\[\tilde P_t:=P_{\int_0^t v_p^2(P_s)\,ds}\]
is a 3-dimensional Bessel process, for which it is well-known that $\tilde\tau^b:=\inf\{t\geq 0:\tilde P_t\geq b\}$
is finite a.s.; since $P$ cannot reach 0, this implies that also $\tau^b=\int_0^{\tilde\tau^b} v_p^2(P_s)\,ds<\infty$ a.s.

Redefining $N_t$ with $\Lambda_t=\Lambda^\ast_t$  and recalling that $v(P_{\tau^b})=v(b)=0$ we thus find, with optional sampling, that
\[v(p)  =N_0 = \E[N_{\tau^b}|\theta=1]
 = \E\left[\left.e^{-r\tau^b}v(P_{\tau^b}) + \int^{\tau^b}_0e^{-rt}d\Lambda^\ast_t\right\vert\theta=1\right]
= \hat{\mathcal J}^2(\tau^b, \Lambda^\ast; p).\]
Thus 
\[v(p)=\hat{\mathcal J}^2(\tau^b, \Lambda^\ast; p)=\sup_{\Lambda\in\mathbb L} \hat{\mathcal J}^2(\tau^b, \Lambda; p),\]
which, in view of Remark~\ref{rem}, completes the proof.
\end{proof}

\begin{remark}
A closer inspection of the proof above shows that the specification of $\lambda^\ast$ on $(b,1)$ is somewhat arbitrary. In fact,
any specification with $\lambda^\ast p\geq rM$ on $(b,1)$ (cf. \eqref{ineqb}) would give a NE $(\tau^b,\Lambda^\ast)$.
\end{remark}

\begin{remark}
In equilibrium, the relation 
\[X_t=v(P_t)-v(p)-r\int_0^t v(P_s)\,ds\]
holds (indeed, this can be checked by applying It\^{o}'s formula to the right hand side and then compare with the equilibrium dynamics of $X$). In the other direction, we have been unable to determine the exact form of the functional mapping of a path of $X$ into a value of $P$.
\end{remark}

%
%
%

%

\section{A Nash equilibrium with randomized stopping} \label{sec:mixed}
In order to find an equilibrium for the remaining case, in which the cost of stopping satisfies $M> \hat M = \frac{\sqrt{\pi}}{2\sqrt{r}}$, we will in this section expand the class of allowed stopping strategies for the account holder to include \textit{randomized} stopping rules.
Thus, instead of looking for an equilibrium with a threshold time $\tau^b$ as in Section~\ref{sec:purestrategies} we here look for an equilibrium with a randomized stopping time specified by an \textit{intensity} $\beta$ with which the account holder stops; in particular, the equilibrium intensity depends on the conditional probability the account holder assigns to the existence of the fraudster. 

\begin{definition}[Randomized stopping time]
Let $\mathcal A$ be the family of right-continuous non-decreasing processes $\Gamma$ which are adapted to the filtration $\F^X$ and with $\Gamma_{0-}=0$. Let $U$ be a random variable which is $\mbox{Exp(1)}$-distributed and independent of all other random sources. A \textit{randomized stopping time} $\gamma$ is a random variable of the form 
$$
\gamma:=\gamma^{\Gamma}:=\inf \{t\geq 0 : \Gamma_t > U\}
$$ 
where $\Gamma\in \mathcal A$. We say that the random time $\gamma$ is \textit{generated} by the process $\Gamma$, and we denote the class of randomized stopping times by $\mathbb T^r$.
\end{definition}

Given a pair $(\gamma,\Lambda)\in (\mathbb T^r, \mathbb L)$, we define as before the expected cost 
 \begin{equation*} 
\mathcal J^1(\gamma,\Lambda;p) = \E \left [\theta \int _0^{\gamma} e^{-rs} d\Lambda_s   + e^{-r\gamma}M\I{\gamma<\infty} \right ]
\end{equation*}
for the account holder, the {\em ex ante} expected payoff 
\begin{equation*}
\mathcal J^2(\gamma,\Lambda;p) =  \E \left [ \theta\int _0^{\gamma} e^{-rs}  d\Lambda_s \right ],
\end{equation*}
for the fraudster, and the expected payoff 
\[\hat{\mathcal J}^2(\gamma,\Lambda;p) =  \E \left [\left. \int _0^{\gamma} e^{-rs}  d\Lambda_s \right\vert\theta=1\right ]\]
for the fraudster in the interim version of the game.

\begin{definition}[Mixed stopping Nash equilibrium]
\label{def:msNE} 
A pair of strategies $(\gamma^\ast, \Lambda^\ast)\in (\mathbb T^r, \mathbb L)$ is a mixed stopping Nash equilibrium (msNE) if
\begin{align*}
&\mathcal J^1(\gamma^\ast, \Lambda^\ast;p) \leq  \mathcal J^1(\gamma, \Lambda^\ast;p) \\
&\mathcal J^2(\gamma^\ast, \Lambda^\ast;p) \geq  \mathcal J^2(\gamma^\ast, \Lambda;p),
\end{align*}
or, equivalently,
\begin{align*}
&\mathcal J^1(\gamma^\ast, \Lambda^\ast;p) \leq  \mathcal J^1(\gamma, \Lambda^\ast;p) \\
& \hat{\mathcal J}^2(\gamma^\ast, \Lambda^\ast;p) \geq  \hat{\mathcal J}^2(\gamma^\ast, \Lambda;p)
\end{align*}
for any $(\gamma,\Lambda)\in(\mathbb T^r, \mathbb L)$. 
\end{definition}

\subsection{Derivation of a mixed stopping Nash equilibrium}
In this section we will look for a msNE in the sense of Definition \ref{def:msNE}. We use heuristic reasoning to obtain a candidate equilibrium which we later verify in Section \ref{sec:verifymsNE}.
As in Section \ref{sec:purestrategies} we consider a process $P$ given by 
\[dP_t=-\lambda^\ast(P_t)P_t(1-P_t)(dX_t+\lambda^\ast(P_t)P_t\,dt)\]
for some non-negative function $\lambda^\ast$ to be determined. We will look for a msNE where the fraud strategy is of the kind \eqref{Markovian-fraud-strat} 
and the stopping strategy is a randomized stopping time generated by a process $\Gamma$ on the form 
$\Gamma_t=\int_0^t\beta(P_s)\,ds$ for some function $\beta$; thus $\beta$ specifies the intensity with which the account holder stops. Moreover, we conjecture that 
$\beta$ has the form
\[\beta(p)=\left\{\begin{array}{ll} 0 & 0<p \leq b\\
\mbox{positive} & b<p<a\\
\infty & p \geq a\end{array}\right.\]
for some $a,b$ with $b<a$, where the infinite intensity should be 
understood as immediate stopping whenever $P_t\geq a$. 
From the fraudster's perspective, martingale arguments as in Section~\ref{sec:purestrategies} again suggest that with
$\Lambda^\ast_t=\int_0^t\lambda^\ast(P_s)\,ds$, the equilibrium value function $v$ should on the interval $(0,a)$ satisfy
\[
\frac{(\lambda^\ast)^2p^2(1-p)^2}{2}v_{pp}-(\lambda^\ast)^2p^2(1-p) v_p  +  \lambda^\ast p (1-p) \lambda v_p -rv-\beta v+\lambda \leq 0\]
for all $\lambda\geq 0$ and 
\begin{equation}
\label{opt}
\frac{(\lambda^\ast)^2p^2(1-p)^2}{2}v_{pp}-(\lambda^\ast)^2p^2(1-p) v_p  +  (\lambda^\ast)^2p (1-p) v_p -rv-\beta v+\lambda^\ast =0,
\end{equation}
cf. \eqref{eq:suboptimality} and \eqref{eq:optimality}, respectively.
Consequently,
\begin{equation}
\label{lambda}
\lambda^\ast p (1-p) v_p +1 =0,
\end{equation}
which inserted in \eqref{opt} leads to a non-linear equation
\begin{equation}
\label{nonlin}
\frac{v_{pp}}{2v_p} -rvv_p -\beta vv_p- \frac{1}{1-p}=0
\end{equation}
on $(0,a)$. 
Specializing to the interval $(0,b)$ where $\beta=0$, we have that
\[ \frac{v_{pp}}{2v_p} -rvv_p - \frac{1}{1-p}=0\]
with general solution
\begin{equation*}
\Phi (\sqrt{2r}v(p)) =\frac{C}{1-p} + D.
\end{equation*}
Imposing the boundary condition $v(0+)=\infty$ yields $C+D=1$ so that, for $0<p<b$, 
\[\Phi(\sqrt{2r}v(p))=\frac{Cp}{1-p}+1 ,\]
or, equivalently,
\begin{equation}
\label{Candb}
\Psi(\sqrt{2r}v(p))=\frac{-Cp}{1-p}.
\end{equation}
However, since the account holder does not stop with certainty at $p=b$, we do \textbf{not} impose $v(b)=0$, so the constant
$C$ remains to be specified.

Next, we expect, as in Section \ref{sec:purestrategies}, that the value function $u$ of the account holder satisfies
\[\frac{(\lambda^\ast)^2p^2(1-p)^2}{2} u_{pp} - r u +p\lambda^\ast  =0 \]
in the interval $(0,b)$ where stopping does not happen. Solving the equation above using \eqref{lambda} and the boundary conditions $u(0)=0$ and $u(b)=M$ gives  
\begin{eqnarray}\label{ueq}
\notag
u(p) &=& A(1-p)  v(p) +B (1-p)  F(\sqrt{2r}v(p)) + v(p)\\
&=& pv(p)+(M-bv(b))\frac{ (1-p)  F(\sqrt{2r}v(p))}{ (1-b)F(\sqrt{2r} v(b))}
\end{eqnarray}
for $0<p\leq b$.

By the indifference principle in game theory, we also expect that
the expected cost for the account holder is constantly equal to $M$ for $b<p<a$ where the intensity of stopping is positive. 
Imposing $u=M$ on the interval $(b,a)$, and using martingale arguments as above, we find that 
$$
0=\frac{(\lambda^\ast)^2p^2(1-p)^2}{2} u_{pp} - r u +p\lambda^\ast  = -rM + p\lambda^\ast  
$$
so that 
\[\lambda^\ast = \frac{rM}{p} \]
for $ b\leq p\leq  a$.
By \eqref{lambda} we thus have
\begin{equation}
\label{der}
v_p = -\frac{1}{rM (1-p)},
\end{equation}
and integration yields 
\[v= \frac{1}{rM} \ln(1-p) + E.\]
Imposing the boundary condition $v(a)=0$ (with the motivation that there is no possibility for the fraudster to act after 
reaching $a$) gives
\begin{equation} 
\label{vinba}
v(p)= \frac{1}{rM} \ln\frac{1-p}{1-a}
\end{equation}
for  $ b\leq p\leq  a$.

Note that in \eqref{ueq}, the free boundary $b$ and the value $v(b)$ (equivalently, $b$ and the constant $C$ in \eqref{Candb}) are still unknown. Imposing the smooth fit condition $u_p(b)=0$ gives 
\[v(b) + bv_p(b) =(M-bv(b))\left(\frac{1}{1-b} + \frac{\sqrt{2r}v_p(b)\Psi(\sqrt{2r}v(b)))}{ F(\sqrt{2r} v(b))} \right),\]
which simplifies to
\begin{equation} \label{lofven}
\frac{M-v(b)}{1-b} + \frac{M\sqrt{2r}\Psi(\sqrt{2r}v(b))-b\varphi(\sqrt{2r}v(b))}{F(\sqrt{2r}v(b))}v_p(b)=0.
\end{equation}
Moreover, imposing $v_p(b-)=v_p(b+)$ we obtain from \eqref{Candb} and \eqref{der} that
\begin{equation} 
v_p(b)=\frac{C}{ \sqrt{2 r}(1-b)^2\varphi{(\sqrt{2 r}v(b))}} =  -\frac{1}{rM(1-b)} . \label{eq:constantseqsys3}
\end{equation}
Using \eqref{Candb} in \eqref{eq:constantseqsys3} we find that 
\begin{equation}
\label{br}
b=\frac{\sqrt rM\Psi(\sqrt{2r}v(b))}{\sqrt{2}\varphi(\sqrt{2r}v(b))},
\end{equation}
so \eqref{lofven} simplifies to 
\begin{equation}
\label{palme}
\frac{M-v(b)}{1-b} + \frac{ M\sqrt{r/2}\Psi(\sqrt{2r}v(b))}{F(\sqrt{2r}v(b))}v_p(b)=0.
\end{equation}
Moreover, inserting the second expression for $v_p(b)$ in \eqref{eq:constantseqsys3} into \eqref{palme} yields 
\[\sqrt{2r}(M-v(b))= \frac{\Psi(\sqrt{2r}v(b))}{F(\sqrt{2r}v(b))},\]
i.e.,
\begin{equation}\label{eqnforv}
\sqrt{2r}(M-v(b))F(\sqrt{2r}v(b)) -\Psi(\sqrt{2r}v(b))=0.
\end{equation}
Note that \eqref{eqnforv} is an equation in the single unknown variable $v(b)$.

\begin{lemma}\label{lem}
Assume that $M>\hat M=\frac{\sqrt\pi}{2\sqrt{r}}$. Then equation \eqref{eqnforv} has 
a unique positive solution $v(b)$. Moreover, $v(b)\in(0,M/2)$.
\end{lemma}

\begin{proof}
Let $f(z):=\sqrt{2r}(M-z)F(\sqrt{2r}z) -\Psi(\sqrt{2r}z)$.
Then $f(0)=\frac{\sqrt{r}M}{\sqrt{\pi}}-\frac{1}{2}>0$ and 
$f(z)<0$ for $z\geq M$, so a solution to $f(z)=0$ exists in $(0,M)$ by continuity. Differentiation yields
$$
f'(z) =2r \Psi(\sqrt{2r} z ) (2z-M).
$$
Thus, $f$ is strictly decreasing up to $M/2$ and increasing for $z > \frac{M}{2}$. Consequently, since 
$f(M) < 0$ the solution must be unique and satisfy $v(b)\in(0, M/2)$.
\end{proof}

Inserting \eqref{der} in \eqref{nonlin} we get 
$$
\beta v \frac{1}{r^2M^2(1-p)^2} = \frac{v_{pp}}{2 } - \frac{v_p}{(1-p)} -r v v_p^2 = \frac{1}{2rM}\frac{1}{(1-p)^2} - r v  \frac{1}{r^2M^2(1-p)^2}
$$
which yields the candidate equilibrium stopping intensity
\begin{equation*} 
\beta(p) = 
\frac{rM}{2v(p)} - r 
\end{equation*}
for $ b<p<a$. Note that since $v$ is decreasing on $(b,a)$ by \eqref{vinba}, 
\[\beta(p)=\frac{rM}{2v(p)} - r \geq \frac{rM}{2v(b)} - r >0\]
by Lemma~\ref{lem}.

\subsection{The candidate msNE} \label{sec:candidate-msNE}
We now summarize by describing our candidate for a msNE and the corresponding value functions for the case 
\begin{equation}
\label{Mbig}
M > \hat M=\frac{\sqrt{\pi}}{2\sqrt r}.
\end{equation}

Let $v(b)$ be the unique positive solution of 
\begin{equation}\label{v(b)}
\sqrt{2r}(M-v(b))F(\sqrt{2r}v(b)) -\Psi(\sqrt{2r}v(b))=0,
\end{equation}
cf. Lemma~\ref{lem}, and let 
\begin{equation}
\label{bdef}
b:=\frac{\sqrt rM\Psi(\sqrt{2r}v(b))}{\sqrt{2}\varphi(\sqrt{2r}v(b))}
\end{equation}
(in Section~\ref{propertiesr} below we verify that $b<1$)
and
\[a:=1-(1-b)e^{-rMv(b)},\]
cf. \eqref{br} and \eqref{vinba}, respectively.
Define $v$ (the candidate equilibrium value function for the fraudster in the interim version of the game) by
\begin{align} \label{v-mixed}
\begin{cases} 
v(p)= \frac{1}{\sqrt{2r}}\Phi^{-1}\left(1- \frac{(1-b)p}{b(1-p)} \left(1-\Phi(\sqrt{2r}v(b))\right)\right) & 0 < p \leq b\\
v(p)= \frac{1}{rM} \ln\frac{1-p}{1-b}    &  b<p \leq a \\
v(p)= 0 & p>a
\end{cases}
\end{align}
and $u$ (the candidate equilibrium value function for the account holder) by
\begin{align} \label{u-mixed}
u(p) = \begin{cases}
pv(p)+(M-bv(b))\frac{ (1-p)  F(\sqrt{2r}v(p))}{ (1-b)F(\sqrt{2r} v(b))} & 0 < p \leq  b \\
 M & p>b.
\end{cases}
\end{align}
Set
\[\beta(p)=\left\{\begin{array}{ll} 0 & 0 < p \leq b\\
 \frac{rM}{2v(p)} - r & b<p<a\\
\infty & p \geq a\end{array}\right.\]
and 
\[\lambda^\ast(p)=\left\{\begin{array}{ll} \frac{-1}{p(1-p)v_p(p)} & 0 < p \leq b\\
\frac{rM}{p}& b<p<1.\end{array}\right.\]
Furthermore, let $(X,P)=(X^{\Lambda},P^{\Lambda})$ for $\Lambda\in \mathbb L$ be defined by
\begin{equation*}
\left\{\begin{array}{ll}
dX_t=-\theta d\Lambda_t + dW_t\\
dP_t=- \lambda^\ast(P_t)P_t (1-P_t) (dX_t + \lambda^\ast(P_t) P_t \,dt) \end{array}\right.
\end{equation*}
with $X_0=0$ and $P_0=p$.
Our candidate msNE $(\gamma^\ast,\Lambda^\ast)$ is then given by 
\begin{align}
\gamma^\ast = \inf	\{t\geq 0:  \Gamma^\ast_t> U\} \label{msNE-stopping-strategy}
\end{align}
with $\Gamma^\ast_t:=\int_0^t\beta(P_s)\,ds$ and $U \sim Exp(1)$, and
\begin{align} 
\Lambda^\ast_t= \int_0^t\lambda^\ast(P_s)\,ds. \label{msNE-fraud-strategy}
\end{align}

\begin{figure}[h]
  \subfigure[]{\includegraphics[width=0.45\textwidth]{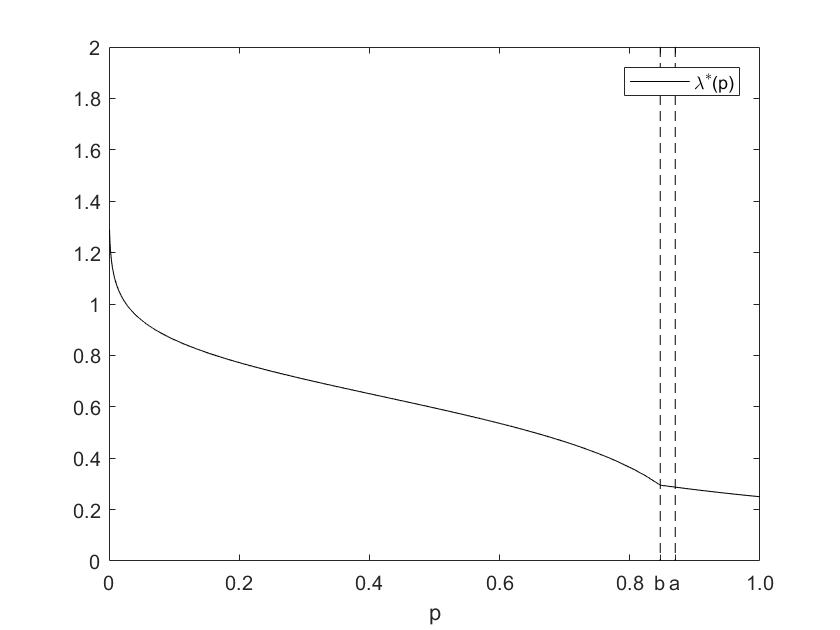}\label{fig:lambda_mixed}}
  \subfigure[]{\includegraphics[width=0.45\textwidth]{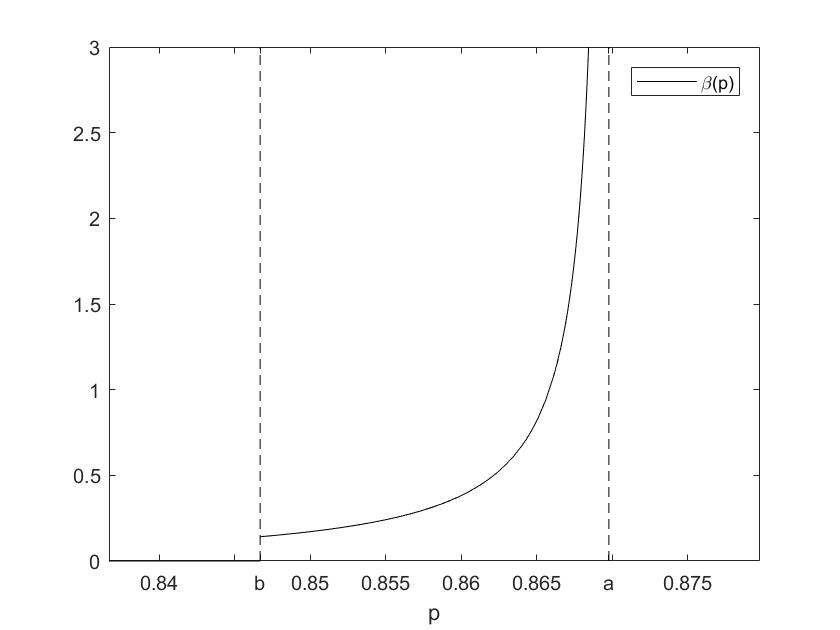}\label{fig:beta}}
\caption{
%
The intensity of the equilibrium fraudster strategy $\lambda^*$ and  
the intensity of the account holder equilibrium stopping strategy $\beta$ for the msNE, cf. Section \ref{sec:candidate-msNE}. The parameters are $r=0.05$ and $M=5$ (note that $M>\hat M\approx 3.96$).}
\end{figure}

For graphical illustrations of $\lambda^\ast$, $\beta$, and the corresponding equilibrium value functions, see Figures~\ref{fig:lambda_mixed}-\ref{fig:beta} and \ref{fig:uv_mixed}.

\subsection{Properties of the candidate msNE}
\label{propertiesr}

Propositions~\ref{thm:existence-etc} and \ref{thm:filter-etc} hold also in the current case.
In this section we show that also $b<1$ and $u\leq M$ hold.

\begin{lemma}
Assume that \eqref{Mbig} holds.
Then the value $b$ defined in \eqref{bdef} satisfies $b<1$.
\end{lemma}

\begin{proof}
From \eqref{br} and \eqref{v(b)} we know that 
\begin{align*}
b=\frac{\sqrt{r}M\Psi(\sqrt{2r}v(b))}{\sqrt 2\varphi(\sqrt{2r}v(b))} =
\frac{ \left(\Psi(\sqrt{2r}v(b))+ \sqrt{2r}v(b)F(\sqrt{2r}v(b))\right) \Psi(\sqrt{2r}v(b))}{2F(\sqrt{2r}v(b))\varphi(\sqrt{2r}v(b))},
\end{align*}
so it suffices to check that
\[\Psi^2(z) + z\varphi(z)\Psi(z)-z^2\Psi^2(z)<2\varphi^2(z)-2z\varphi(z)\Psi(z),\]
or, equivalently, 
\begin{equation}
\label{mill}
2\frac{\varphi^2(z)}{\Psi^2(z)} - 3z\frac{\varphi(z)}{\Psi(z)}+ z^2>1
\end{equation}
for all $z>0$.
However, \eqref{mill} is a well-known inequality that holds for all $z\in\mathbb R$;
for completeness, we provide the following argument from \cite{S}.
Define
\[f(z):=2\frac{\varphi^2(z)}{\Psi^2(z)} - 3z\frac{\varphi(z)}{\Psi(z)}+ z^2=\left(\frac{\varphi(z)}{\Psi(z)}-z\right)\left(2\frac{\varphi(z)}{\Psi(z)}-z\right)\]
and note that $f(z)\to \infty$ as $z\to-\infty$ and $f(z)\to 1$ as $z\to\infty$ since 
\[\frac{\varphi(z)}{\Psi(z)}=z +z^{-1} + o(z^{-1})\]
as $z\to\infty$.
Moreover, 
\begin{eqnarray}
\label{fprop}
f'(z) &=& -2\left(\frac{\varphi(z)}{\Psi(z)}-z\right)\left(1+z\frac{\varphi(z)}{\Psi(z)}-\frac{\varphi^2(z)}{\Psi^2(z)}\right)-
\frac{\varphi(z)}{\Psi(z)}(1-f(z))\\
\notag
&<& - \frac{\varphi(z)}{\Psi(z)}(1-f(z)).
\end{eqnarray}
If $f(z_0)=1$ for some $z_0\in\mathbb R$, then there exists a $z\geq z_0$ with $f(z)\leq 1$ and 
$f'(z)=0$, which contradicts \eqref{fprop}. Consequently, $f>1$, which proves \eqref{mill}.
\end{proof}

\begin{figure}[htp!]\centering
\includegraphics[width=0.5\textwidth]{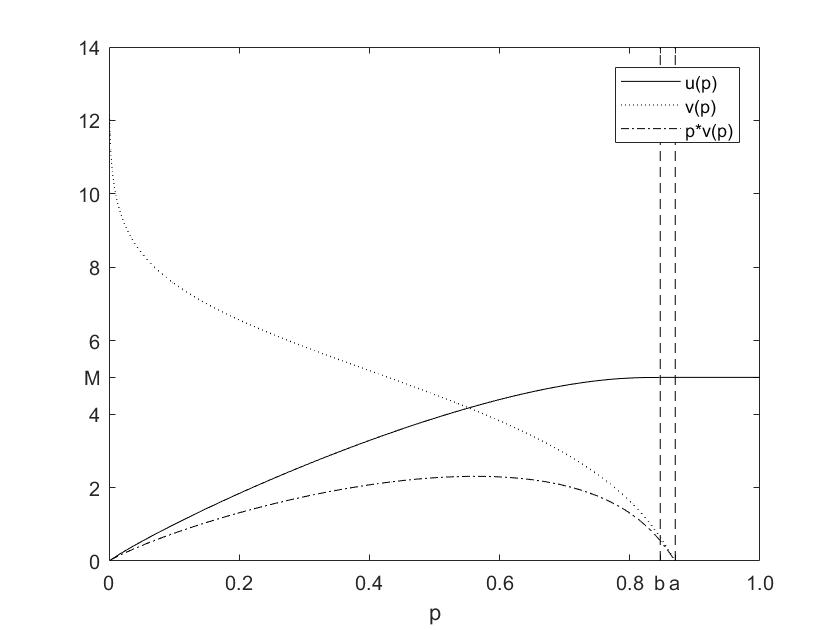}
\caption{
The equilibrium value functions $u$ (account holder) and $v$ (fraudster, interim version of the game) for the msNE, cf. Section \ref{sec:candidate-msNE}. We also include the fraudster equilibrium value function $pv$ (msNE) in the ex ante game, cf. Remark~\ref{rem}.
The parameters are as in Figures~\ref{fig:lambda_mixed}-\ref{fig:beta}.}
\label{fig:uv_mixed}
\end{figure}

\begin{lemma}
The function $u$ defined in \eqref{u-mixed} (with $v$ defined in \eqref{v-mixed}) satisfies $u\leq M$.
\end{lemma}

\begin{proof}
As in the proof of Lemma~\ref{ubound} it suffices to show that  $\lambda^\ast p-ru\geq 0$ for $p\in(0,b)$ and we therefore assume that $p\in(0,b)$ in this proof.
By definition of $v(b)$, $v$ and $F$ we have that 
$$\frac{\Psi(\sqrt{2r}v(p))}{\Psi(\sqrt{2r}v(b))} = \frac{1-b}{1-p}\frac{p}{b}$$  
and
$$\frac{F(\sqrt{2r}v(p))}{F(\sqrt{2r}v(b))} = 
\sqrt{2r}(M-v(b))\frac{\varphi(\sqrt{2r}v(p))-\sqrt{2r}v(p)\Psi(\sqrt{2r}v(p))}{\Psi(\sqrt{2r}v(b))}.$$ 
We obtain 
\begin{eqnarray*}
u(p)  &=& pv(p)+(M-bv(b))\frac{ (1-p)  F(\sqrt{2r}v(p))}{ (1-b)F(\sqrt{2r} v(b))}\\
& =& pv(p)+D\frac{1-p}{1-b}\frac{\varphi(\sqrt{2r}v(p))-\sqrt{2r}v(p)\Psi(\sqrt{2r}v(p))}{\Psi(\sqrt{2r}v(b))}\\
& =& pv(p)+D\frac{1-p}{1-b}\frac{\varphi(\sqrt{2r}v(p))}{\Psi(\sqrt{2r}v(b))}-D\frac{p}{b}\sqrt{2r}v(p)\\
& =& pv(p)\left(1-\sqrt{2r}\frac{D}{b}\right)+\frac{(1-p)\varphi(\sqrt{2r}v(p))D}{(1-b)\Psi(\sqrt{2r}v(b))},
\end{eqnarray*}
where $D:=\sqrt{2r}(M-v(b))(M-bv(b))$. 

Using the definition of $\lambda^\ast$ and differentiation of the expression for $v(p)$ we also find that
\begin{align*}
\lambda^\ast(p)p = \frac{-1}{(1-p)v_p(p)} = \frac{\sqrt{2r}(1-p)b\varphi(\sqrt{2r}v(p))}{(1-b)\Psi(\sqrt{2r}v(b))}.
\end{align*}
Hence,
\begin{eqnarray}
\label{expr}
 \lambda^\ast(p)p -ru(p)
& =& pv(p)r\left(\sqrt{2r}\frac{D}{b}-1\right)+(1-p)\varphi(\sqrt{2r}v(p))\frac{\sqrt{2r}b-rD}{(1-b)\Psi(\sqrt{2r}v(b))}\\
\notag & =& A\left(pv(p)+(1-p)\varphi(\sqrt{2r}v(p))B\right),
\end{eqnarray}
where $A:=r\left(\sqrt{2r}\frac{D}{b}-1\right)$ and 
\[B:=\frac{\sqrt{2r}b-rD}{(1-b)\Psi(\sqrt{2r}v(b))A}.\]
Since $\lambda^\ast(b)b = ru(b)= rM$, it follows that \eqref{expr} evaluated at $p=b$ equals $0$, so that
\begin{align*}
\lambda^\ast(p)p -ru(p) = A\left(pv(p)-(1-p)\varphi(\sqrt{2r}v(p))\frac{bv(b)}{(1-b)\varphi(\sqrt{2r}v(b))}\right)
\end{align*}

Next,
\begin{eqnarray*} 
bA/r &=& \sqrt{2r}D-b\\
&=& 2r(M-v(b))(M-bv(b))- b \\
&=& \sqrt{2r}\frac{\Psi(\sqrt{2r}v(b))}{F(\sqrt{2r}v(b))}\left(M-   \frac{\sqrt rM\Psi(\sqrt{2r}v(b))}{\sqrt{2}\varphi(\sqrt{2r}v(b))}   v(b)\right)- \frac{\sqrt rM\Psi(\sqrt{2r}v(b))}{\sqrt{2}\varphi(\sqrt{2r}v(b))} \\
&=& \sqrt{2r}\frac{\Psi(\sqrt{2r}v(b))}{F(\sqrt{2r}v(b))}\left( M-   \frac{\sqrt rM\Psi(\sqrt{2r}v(b))}{\sqrt{2}\varphi(\sqrt{2r}v(b))}   v(b) -\frac{MF(\sqrt{2r}v(b))}{2\varphi(\sqrt{2r}v(b))}\right)\\
&=& \sqrt{2r}\frac{M\Psi(\sqrt{2r}v(b))}{2F(\sqrt{2r}v(b))}>0,
\end{eqnarray*}
so $A>0$. Thus we are done if we can show that 
\[G(p):=\frac{pv(p)}{(1-p)\varphi(\sqrt{2r}v(p))}\]
is a decreasing function (since this implies that $\lambda^\ast(p)p -ru(p)\geq 0$ for $0<p<b$).
To do that, differentiate to get
\begin{eqnarray*}
(1-p)^2\varphi^2(\sqrt{2r}v(p)) G_p(p) &=& v(p)\varphi(\sqrt{2r}v(p)) +p(1-p)\varphi(\sqrt{2r}v(p))\left( 1 +2rv^2(p)\right)v_p(p)\\
&=& v(p)\varphi(\sqrt{2r}v(p))-\left( 1 +2rv^2(p)\right)\frac{\Psi(\sqrt{2r}v(p))}{\sqrt{2r}},
\end{eqnarray*}
which is negative since $x\varphi(x)\leq (1+x^2)\Psi(x)$ for all $x$.
\end{proof}

\subsection{Verification for the mixed stopping Nash equilibrium}
\label{sec:verifymsNE}
In this section we prove that the candidate msNE $(\gamma^\ast,\Lambda^\ast)$ defined in Section \ref{sec:candidate-msNE} is indeed a msNE.

\begin{theorem}
\label{msNE}
{\bf (Verification.)} Assume that the account holder's cost for stopping satisfies $M>\hat M$. Then, the pair $(\gamma^\ast,\Lambda^\ast)$ defined in \eqref{msNE-stopping-strategy} and \eqref{msNE-fraud-strategy} is a msNE in the sense of Definition \ref{def:msNE}. Moreover, the corresponding equilibrium value functions for the account holder and the fraudster (in the interim version of the game) are given by \eqref{u-mixed} and \eqref{v-mixed}, i.e.,
\[
u(p)= \mathcal {\mathcal J}^1(\gamma^\ast,\Lambda^\ast;p)
\mbox{ and } 
v(p) = \hat{\mathcal J}^2(\gamma^\ast,\Lambda^\ast;p).
\]
\end{theorem}

\begin{proof}
\textbf{Optimality of $\gamma^\ast$.}
First note that 
\begin{equation}
\label{ineqmixed}
\frac{(\lambda^\ast)^2 p^2(1-p)^2}{2} u_{pp} -r u +\lambda^\ast p= 0
\end{equation}
for $p\in(0,1)\setminus\{b\}$. In fact, \eqref{ineqmixed} holds with equality on $(0,b)$ by construction, and for $p\geq b$ we have $u\equiv M$ so
\[\frac{(\lambda^\ast)^2 p^2(1-p)^2}{2} u_{pp} -r u +\lambda^\ast p=-rM+rM=0.\]
Consequently, with $\Lambda^\ast$ fixed, It\^{o}'s formula gives 
\[e^{-rt}u(P_t)=u(p)- \int_0^t e^{-rs}\lambda^\ast(P_s)P_s\,ds+\int_0^t e^{-rs}\frac{u_p(P_s)}{v_p(P_s)}\,d\hat W_s,\]
where we note that the stochastic integral is a martingale.
Now fix an $\F^X$-stopping $\tau\in\mathbb T$. By optional sampling, 
\begin{eqnarray*}
u(p) &=& \E\left[e^{-r\tau\wedge t}u(P_{\tau\wedge t})+ \int_0^{\tau\wedge t} e^{-rs}\lambda^\ast(P_s)P_s\,ds\right] \\
&=& \E\left[e^{-r\tau\wedge t}u(P_{\tau\wedge t})+ \theta\int_0^{\tau\wedge t} e^{-rs}\lambda^\ast(P_s)\,ds\right] \\
&\leq& \E\left[e^{-r\tau\wedge t}M+ \theta\int_0^{\tau\wedge t} e^{-rs}\lambda^\ast(P_s)\,ds\right],
\end{eqnarray*}
where the second equality follows using iterated expectations since $P_t  =\E\left[\theta \vert \,  \F^X_t\right]$, see Proposition~\ref{thm:filter-etc}.
Letting $t\to\infty$ gives 
\[u(p)\leq \E\left[e^{-r\tau}M+ \theta\int_0^{\tau} e^{-rs}\lambda^\ast(P_s)\,ds\right]=\mathcal J^1(\tau,\Lambda^\ast;p)\]
by monotone convergence. 

Now consider a randomized stopping time $\gamma\in\mathbb T^r$ generated by $\Gamma\in\mathbb A$, 
and let $\gamma(c):= \inf\{t\geq 0 : \Gamma_t > c\}$ for $c \in [0, \infty)$ (so that $\gamma=\gamma(U)$).
Then $\gamma(c)\in\mathbb T$, and 
\[\mathcal J^1(\gamma,\Lambda^\ast;p) = \int^\infty_0 e^{-c}\mathcal J^1(\gamma(c),\Lambda^\ast;p)\,dc\geq \int^\infty_0 e^{-c}u(p)\,dc = u(p),\]
so
\[u(p)\leq \inf_{\gamma\in\mathbb T^r}\mathcal J^1(\gamma,\Lambda^\ast;p).\]

For the reverse inequality, consider $\gamma=\gamma^\ast$. Since $u\left(P_{\gamma^\ast(c)}\right)=M$ on the event $\{\gamma^\ast(c)<\infty\}$,
we have for $c\geq 0$ that
\begin{eqnarray*}
u(p) &=& \E\left[e^{-r\gamma^\ast(c)\wedge t}u(P_{\gamma^\ast(c)\wedge t})
+ \int_0^{\gamma^\ast(c)\wedge t} e^{-rs}\lambda^\ast(P_s)P_s\,ds\right]\\
&\geq& \E\left[e^{-r\gamma^\ast(c)}M1_{\{\gamma^\ast(c)\leq t\}}
+ \theta\int_0^{\gamma^\ast(c)\wedge t} e^{-rs}\lambda^\ast(P_s)\,ds\right]\\
&\to& 
\E\left[e^{-r\gamma^\ast(c)}M+ \int_0^{\gamma^\ast(c)} e^{-rs}\lambda^\ast(P_s)\,ds\right]  = \mathcal J^1(\gamma^\ast(c),\Lambda^\ast;p)
\end{eqnarray*}
as $t\to\infty$ by monotone convergence. Consequently, 
\[\mathcal J^1(\gamma^\ast,\Lambda^\ast;p) = \int^\infty_0 e^{-c}\mathcal J^1(\gamma^\ast(c),\Lambda^\ast;p)\,dc \leq u(p),\]
so 
\[u(p)=\mathcal J^1(\gamma^\ast,\Lambda^\ast;p) =\inf_{\gamma\in\mathbb T^r}\mathcal J^1(\gamma,\Lambda^\ast;p).\]

{\bf Optimality of $\Lambda^\ast$.}
Let us now fix the stopping strategy $\gamma^\ast$ generated by 
$$\Gamma_t  = \int_0^t \beta(P_s) \,ds,$$ 
where 
\[\left\{\begin{array}{ll}
dX_t=-\theta d\Lambda_t + dW_t\\
dP_t=- \lambda^\ast(P_t)P_t (1-P_t) (dX_t + \lambda^\ast(P_t) P_t \,dt). \end{array}\right.\]
If the fraudster uses a strategy $\Lambda$, the process $P$ satisfies
\begin{equation*} 
dP_t = -\lambda^\ast(P_t)P_t (1-P_t) (-d\Lambda_t+ \lambda^\ast(P_t)P_t\, dt) - \lambda^\ast(P_t) P_t (1-P_t )\, dW_t
\end{equation*}
on the event $\{\theta=1\}$.
We then want to show that $ \Lambda^\ast$ is an optimal response for the fraudster. More precisely, we want to show 
$$
\hat{\mathcal J}^2(\gamma^\ast,\Lambda^\ast;p) :=  \E \left [ \int _0^{\gamma^\ast} e^{-rs} d\Lambda_s \vert\theta=1 \right ] \geq \hat{\mathcal J}^2(\gamma^\ast,\Lambda;p) 
$$
for any fraud strategy $\Lambda\in\mathbb L$.

To do that, first define a process $N$ by 
\[N_t:=e^{-rt}v(P_{t})(1-Q_t)+ \int^{t\wedge\gamma^\ast}_0e^{-rs}d\Lambda_s,\]
where 
\[Q_t:=1_{\{t\geq \gamma^\ast\}}\]
is a jump process with intensity $\beta(P_t)$.
Note that 
\begin{eqnarray*}
dN_t &=&  1_{\{t\leq\gamma^\ast\}}e^{-rt}\left(\frac{1}{2}(\lambda^\ast(P_t))^2P_t^2(1-P_t)^2v_{pp}(P_t) -(\lambda^\ast(P_t))^2P_t^2(1-P_t)v_{p}(P_t) -rv(P_t)\right)dt\\
&& + 1_{\{t\leq\gamma^\ast\}}e^{-rt}\left(1+\lambda^\ast(P_t)P_t(1-P_t)v_p(P_t)\right)\,d\Lambda_t-e^{-rt}v(P_t)\,dQ_t + 1_{\{t\leq\gamma^\ast\}}e^{-rt}\,dW_t \\
&=& 1_{\{t\leq\gamma^\ast\}}e^{-rt} v(P_t)  (\beta(P_t) \,dt-dQ_t) + 1_{\{t\leq\gamma^\ast\}}e^{-rt}\,dW_t
\end{eqnarray*}
so $N$ is a lower bounded local martingale. Consequently, 
\[v(p)=N_0\geq \E [N_t]\geq \E\left[\int^{t\wedge\gamma^\ast}_0e^{-rs}d\Lambda_s\right],\]
so by monotone convergence we find that $v(p)\geq \sup_{\Lambda}\hat{\mathcal J}^2(\gamma^\ast,\Lambda;p)$.

On the other hand, consider the case $\Lambda =\Lambda^\ast$. 
Note that $Z_t:=e^{-rt}v(P_t)+ \int^{t}_0e^{-rs}d\Lambda^\ast_s$ satisfies
\[dZ_t = e^{-rt}\beta(P_t)v(P_t)\,dt -e^{-rt}\,dW_t\]
for $t<\tau_a:=\inf\{t\geq 0:P_t\geq a\}$.
Since $\beta v$ is bounded, we thus have 
\[Z_t\leq v(p) + C + \int_0^t e^{-rs}\,dW_s\]
for some $C>0$, which shows that $\{Z_t,t\leq \tau_a\}$ is uniformly integrable. Consequently, 
\[N_t:=e^{-rt}v(P_{t})(1-Q_t)+ \int^{t\wedge\gamma^\ast}_0e^{-rs}d\Lambda^\ast_s\]
is a martingale, and optional sampling gives
\[v(p)=\E [N_t]= \E\left[ e^{-rt}v(P_{t})1_{\{t<\gamma^\ast\}}+ \int^{t\wedge\gamma^\ast}_0e^{-rs}d\Lambda^\ast_s \right].\]
The arguments used in Theorem~\ref{main} show that $\gamma^\ast<\infty$ a.s. Consequently, by uniform integrability,
\begin{equation*}
\label{transv}
\lim_{t\to\infty}\E\left[e^{-rt}v(P_t)1_{\{t<\gamma^\ast\}}\right]=0.
\end{equation*}
Therefore, monotone convergence yields 
\[v(p)=\E\left[\int^{\gamma^\ast}_0e^{-rs}d\Lambda^\ast_s\right]= \hat{\mathcal J}^2(\gamma^\ast,\Lambda^\ast;p).\]
Thus 
\[ \hat{\mathcal J}^2(\gamma^\ast,\Lambda^\ast;p)=v(p)\geq \sup_{\Lambda}\hat{\mathcal J}^2(\gamma^\ast,\Lambda;p),\]
which completes the proof.
\end{proof}

\begin{remark}
As in the case with a pure NE described in Section~\ref{sec:pure}, the specification of $\lambda^\ast$ on $(a,1)$ can be done in several ways; we have chosen to use the same formula on $(a,1)$ as on $(b,a)$.
\end{remark}

\end{document}